\documentclass{amsart}
\usepackage{amsmath,amssymb,amscd,amsthm,verbatim,alltt,amsfonts,array}
\usepackage[english]{babel}
\usepackage{latexsym}
\usepackage{amssymb}
\usepackage{euscript}
\usepackage{graphicx}
\usepackage[all]{xy}
\usepackage{hyperref}
\usepackage{enumerate}

\newtheorem{theorem}{Theorem}[section]
\theoremstyle{plain}

\newtheorem{corollary}[theorem]{Corollary}

\newtheorem{definition}[theorem]{Definition}
\newtheorem{example}[theorem]{Example}

\newtheorem{lemma}[theorem]{Lemma}

\newtheorem{proposition}[theorem]{Proposition}
\newtheorem{remark}[theorem]{Remark}

\numberwithin{equation}{section}

\usepackage{algorithm}
\usepackage{algorithmic}

\usepackage{url,doi}

\usepackage{tikz}
\usetikzlibrary{patterns, arrows, mindmap, snakes, decorations.pathreplacing}
\usetikzlibrary{arrows,automata}
\usetikzlibrary{positioning}

\newcommand{\F}{\mathbb{F}}

\def\overset#1#2{
 \mathrel{\mathop{\kern 0pt#2}\limits^{#1}}}

\def\underset#1#2{
 \mathrel{\mathop{\kern 0pt#2}\limits_{#1}}}

\newlength{\resto}
\newcommand{\lres}[2]%
 {%
 \settoheight{\resto}{$#1$}%
 \addtolength{\resto}{-1pt}%
 \raisebox{\resto}{\scriptsize $#2$}\overline{#1}%
 }
\newcommand{\rres}[2]%
 {%
 \settoheight{\resto}{$#1$}%
 \addtolength{\resto}{-1pt}%
 \overline{#1}\raisebox{\resto}{\scriptsize $#2$}%
 }

\newcommand{\N}{\mathbb{N}}

\newcommand{\ff}{\mathbb{F}}

\newcommand{\ls}[1]{(\!(#1)\!)}
\newcommand{\s}[1]{[\![#1]\!]}

\begin{document}

% first the title is needed
\title{Skew Laurent series and general cyclic convolutional codes}
\thanks{Research funded  by the IMAG–Mar\'{\i}a de Maeztu grant
(CEX2020-001105-M) and
``Proyecto PID2023-149565NB-100" from  MICIU/AEI /10.13039/501100011033 and FEDER, UE". \\
We would like to thank the anonymous reviewer for their constructive comments, which prompted the addition of the last section.}

% a short form should be given in case it is too long for the running head

\author[J. G\'omez-Torrecillas \and J. P. S\'anchez-Hern\'andez]{Jos\'e G\'omez-Torrecillas$^1$ \and Jos\'e Patricio S\'anchez-Hern\'andez$^2$}
\address{$^1$Departamento de \'Algebra e IMAG,
Universidad de Granada,
E-18071 Granada, Spain.
$^2$Facultad de Ciencias Exactas,
Universidad Juárez del Estado de Durango,
Durango, México.}
\email{\tt  $^1$gomezj@ugr.es
$^2$jose.sanchez@ujed.mx}

\maketitle

\begin{abstract}
Convolutional codes were originally conceived as vector subspaces of a finite-dimensional vector space over a field of Laurent series having a polynomial basis. Piret and Roos modeled cyclic structures on them by adding a module structure over a finite-dimensional algebra skewed by an algebra automorphism. These cyclic convolutional codes turn out to be equivalent to some right ideals of a skew polynomial ring built from the automorphism. When a skew derivation is considered, serious difficulties arise in defining such a skewed module structure on Laurent series. We discuss some solutions to this problem which involve a purely algebraic treatment of the left skew Laurent series built from a left skew derivation of a general coefficient ring, when possible. 
\end{abstract}

\section*{Introduction}

Let \(\F\ls{X}\) be the Laurent formal power series field with coefficients in a field \(\F\). Classically, see e.g. \cite{Roos:1979}, \(\F\)--linear convolutional codes of length \(n\) may be conceived as  \(\F\ls{X}\)--vector subspaces \(\mathcal{D}\) of \(\F^n\ls{X}\) generated by vectors with entries in the rational field \(\F(X)\) or, equivalently, with an \(\F^n\ls {X} \)--basis of elements of \(\F^n[X]\). In this framework, the indeterminate \(X\) is interpreted as the shift operator acting on blocs \(v_i \in \F^n\) as \(Xv_i = v_{i-1}\).  
It turns out that the map that sends \(\mathcal{D}\) to \(\mathcal{D} \cap \F^n[X]\) is a bijection between  convolutional codes and \(\F[X]\)--submodules of the free \(\F[X]\)--module \(\F^n[X]\) that are direct summands. 

Cyclic structures on convolutional codes were initially introduced by adding a module structure on \(\F^n\ls{X}\). More concretely, identifying vectors of \(\F^n\) with polynomials in \(\F[t]\) of degree up to \(n-1\), and taking the quotient ring  \(A = \F[t]/\langle t^n - 1\rangle \), then a right\footnote{In \cite{Roos:1979}, a left \(A\)--module structure is considered. We prefer  to twist left and right hands in our approach since our references on skew derivations fit better this equivalent formalism.} \(A\)--module structure on \(\F^n\ls{X}\) is defined. To avoid trivial cases, as noted in \cite{Piret:1976}, such a module structure is twisted by an \(\F\)--algebra automorphism \(\sigma\) of \(A\). Then, a convolutional code is cyclic if and only if it is a right \(A\)--submodule of \(\F^n\ls{X}\) (\cite[Theorem 3]{Roos:1979}). 

Under the aforementioned bijection, cyclic convolutional codes correspond to right ideals of \(A[X;\sigma]\) that are \(\F[X]\)--direct summands (see \cite[Proposition 2]{Gomez/Sanchez:2024} for a more general statement). Here, \(A[X;\sigma]\) is the left skew polynomial ring defined from \(\sigma\). Indeed, this is the way in which cyclic convolutional codes are defined in \cite{Gluesing/Schmale:2004}. More generally, cyclic convolutional
codes are defined in [6], under the left-right symmetric equivalent
formulation, as left ideals of the right skew polynomial ring  built
from a right skew derivation of any finite-dimensional \(\F\)-algebra A.  When trying to return to Roos' viewpoint, a major difficulty is that no obvious way to give a right \(A\)--module structure on \(A\ls{X}\), twisted by \((\sigma,\delta)\), is available. For \(\delta = 0\), giving such a structure, as in \cite{Roos:1979}, is equivalent to endow \(A\ls{X}\) of a ring structure extending that of \(A[X;\sigma]\), which is always possible if \(\sigma\) is an automorphism, namely, the left skew Laurent power series ring \(A\ls{X;\sigma}\). When \(\delta \neq 0\), serious difficulties arise, being a traditional solution in the realm of Ring Theory, to consider inverse skew Laurent formal power series ring (see, e.g. \cite[Section 2.3]{Cohn:1995}). However, this object does not fit in the described picture of cyclic convolutional codes.  

Although the ring of  left skew Laurent formal power series \(A\ls{X;\sigma,\delta}\) could not, with the desired properties, exist for a general \(\delta \neq 0\), suitable restrictions on the left skew derivation \((\sigma,\delta)\) allows to define a well founded product extending that of \(A[X;\sigma,\delta]\). Thus, in \cite{Schneider/Venjakob:2010}, under the hypotheses that \(\sigma\) is an automorphism and both \(\delta, -\delta \sigma^{-1}\) are nilpotent, besides additional topological conditions on the ring \(A\) and \((\sigma,\delta)\), such a ring structure is modeled. On the other side, under the condition  \(\delta \sigma = q \sigma \delta\) for some \(q \in \F\) for \(A\) an \(\F\)--algebra, the ring of left skew formal power series \(A\s{X;\sigma,\delta}\) is defined in \cite{Bergen/alt:2011} whenever \(\delta\) is locally nilpotent. For a general coefficient ring \(A\), and with no condition apart from being locally nilpotent, a well defined product for formal power series with coefficients in \(A\) is obtained in \cite{Greenfeld/alt:2019}. 

In \cite{Gomez/Sanchez:2024}, for \(\sigma\) an algebra automorphism of a finite-dimensional \(\F\)--algebra \(A\)  and \(\F^n\) endowed with the structure of a right \(A\)--module, \((\sigma,A)\)--cyclic convolutional codes are defined, in the spirit and generalizing Roos' notion, as right \(A\)--submodules of \(\F^n\ls{X}\). The bijective correspondence between these \((\sigma,A)\)--cyclic convolutional codes and right \(A[X;\sigma,\delta]\)--submodules of \(\F^n[X]\) that are \(\F[X]\)--direct summands is proved in \cite[Proposition 2]{Gomez/Sanchez:2024}. With the aim of exploring to what extent this correspondence still is possible for a left skew derivation \((\sigma,\delta)\) on \(A\), we analyzed how the skew product defined in \cite{Greenfeld/alt:2019} could be extended to Laurent series. Although we are mainly interested, in the realm of convolutional codes, in the case of a finite-dimensional algebra \(A\), our analysis is done on a left skew derivation on a general ring \(A\). We think the outcome could be of independent interest, and it is explained in Section \ref{skewlaurentseries}. 

In order to extend the notion of \((\sigma,A)\)--cyclic convolutional code to that of \((\sigma,\delta,A)\)--cyclic convolutional code, when possible, Section \ref{modulos} collects the needed technical facts on modules over \(A\s{X;\sigma,\delta}\)  and \(A\ls{X;\sigma,\delta}\). 

In Section \ref{biyeccion} we define, starting from a right \(A\)--module structure on \(\F^n\), when possible, a general concept of cyclic convolutional code modeled by a left skew derivation \((\sigma,\delta)\) on a finite-dimensional \(\F\)--algebra \(A\). We call these convolutional codes \((\sigma,\delta,A)\)--cyclic and, when \(\delta = 0\), they boil down to the \((\sigma,A)\)--cyclic convolutional codes from \cite{Gomez/Sanchez:2024} and, ultimately, to cyclic convolutional codes in the sense of Roos. The bijective correspondence between \((\sigma,\delta,A)\)--cyclic convolutional codes and \(\F[X]\)--direct summands that are right \(A[X;\sigma,\delta]\)--submodules of \(\F^n[X]\) is proved in Theorem \ref{correspondencia}. As a consequence, we deduce that the notions of \((\sigma,\delta,A)\)--cyclic convolutional code and (right) ideal code from \cite{Lopez/Szabo:2013} are equivalent, whenever \(\delta\) and \(\delta'\) are nilpotent. 

Section \ref{cabras} briefly discusses the connections between the concepts introduced in this paper and the notion of convolutional codes considered in recent developments. We include a description of how the cyclicity of a convolutional code determines the form of its minimal encoder. Finally, some remarks are provided regarding the existence of an idempotent generator for a cyclic convolutional code (where applicable) and the use of this property to analyze its free distance.

\section{Notation for formal power series}\label{Series}
Let \(M\) be a left module over a ring \(A\) and consider the left \(A\)--module \(M^\mathbb{Z}\) of all maps from \(\mathbb{Z}\) to \(M\). We will use the sequence-like notation \((u_i)_{i \in \mathbb{Z}}\) for these maps. 

In the realm of convolutional codes, 
representing these sequences as formal power series 
\[
\sum_{i \in \mathbb{Z}}u_iX^i= \sum_{i = -\infty}^{\infty} u_iX^i,
\] 
in an indeterminate \(X\) turns out to be useful.  Indeed, the \emph{delay operator} \(X\) on \(M^\mathbb{Z}\) is  defined by the rule \(X(u_i)  = u_{i-1}\). This is obviously an \(A\)--linear invertible operator.  

Under this formalism, the left \(A\)--module \(M^\mathbb{Z}\) will be denoted by \(M\ls {X,X^{-1}}\). 

Some of its relevant submodules are the formal Laurent series
\[
M\ls{X} = \left\{ \sum_{i = i_0}^\infty u_iX^i : i_0\in \mathbb{Z}\right\},
\]
the Laurent polynomials
\[
M[X,X^{-1}] = \left\{\sum_{i = i_0}^{i_1}u_i X^i : i_0, i_1 \in \mathbb{Z}, i_0\leq i_1\right\},
\]
and the polynomials
\[
M[X] = \left\{\sum_{i=0}^n u_iX^i : n \in \mathbb{N}\right\},
\]
all of them with coefficients in \(M\). We will also consider the \(A\)--submodules of all formal power series
\[
M\s{X}  = \left\{\sum_{i = 0}^\infty u_iX^i \right\}.
\]

With \(A\) a commutative ring and \(M = A\),  all these \(A\)--modules, with the exception of  \(A\ls{X,X^{-1}}\), become commutative rings in the usual way. Even in this case, there are  other possible structures of ring on these additive groups. So, in general, we do not assume a priori any ring structure on them.

\section{Skew Laurent formal power series}\label{skewlaurentseries}

A left skew derivation of a ring \(A\) is a pair \((\sigma,\delta)\) of additive maps such that \(\sigma\) is a ring endomorphism of \(A\) and \(\delta\) satisfies the equality \(\delta(ab) = \sigma(a)\delta(b) + \delta(a)b\) for all \(a,b\in A\). A right skew derivation over \(A\) is just a left skew derivation over the opposite ring \(A^{'}\) of \(A\). 

Given a left skew derivation \((\sigma,\delta)\) over \(A\), the free left \(A\)--module \(A[X]\) with basis \(\{X^i : i \in \mathbb{N}\}\) is endowed with a  ring structure whose product is determined by the rules
\begin{equation}\label{skewrule}
X^{i+j} = X^iX^j, \quad Xa = \sigma(a)X + \delta(a), \qquad (i,j \in \mathbb{N}, a \in A). 
\end{equation}

This ring is denoted by \(A[X;\sigma,\delta]\) and their elements are called (left) skew polynomials, and it is referred to as the skew polynomial ring of \((\sigma,\delta)\). A symmetric construction applies for any right skew derivation \((\sigma',\delta')\) of \(A\). We will use the notation \([X;\sigma',\delta']A\) in this case.  

Given any additive map \(L : A \to A\) and \(f = \sum_i f_i X^i \in A[X;\sigma,\delta]\), set
\(L(f) = \sum_i L(f_i)X^i\). We easily get
\begin{equation}\label{skewrule2}
Xf = \sigma(f)X + \delta(f). 
\end{equation}

The multiplication rule \eqref{skewrule2} extends to
\begin{equation}\label{Xnf}
X^nf = \sum_{k=0}^n N_k^{n}(f)X^k,
\end{equation}
for all \(f \in A[X;\sigma,\delta]\), 
for additive endomorphisms of \(A[X;\sigma,\delta]\) defined recursively, for \(0 \leq i \leq n\), by
\[
N_{i}^{n+1} = \sigma N_{i-1}^n + \delta N_i^n, \qquad (N_{-1}^n = 0, N_{0}^0 = id_{A[X;\sigma,\delta]}, N_{n+1}^n = 0). 
\]
It follows from \eqref{Xnf} that 
\begin{equation}\label{lr}
\sum_{i=0}^n X^ia_i = \sum_{i=0}^n \left(\sum_{j=i}^nN_{i}^j(a_j)\right)X^i, 
\end{equation}
for all \(a_0\dots,a_n \in A\), and
\begin{equation}\label{ga}
\left(\sum_{i = 0}^ng_iX^i\right)f = \sum_{i=0}^n \left(\sum_{k= i}^n g_kN_i^k(f)\right)X^i,
\end{equation}
for all \(g_0, \dots, g_n \in A, f \in A[X;\sigma,\delta]\).

When \(\sigma\) is an automorphism,  \((\sigma^{-1},-\delta\sigma^{-1})\) is a right skew derivation, which will be denoted by \((\sigma',\delta')\).  
 Observe that, for all \(f \in A[X;\sigma,\delta]\), one deduces from \eqref{skewrule2} that
\begin{equation}\label{primas}
fX = X\sigma'(f) + \delta'(f),
\end{equation}
for all \(f \in A[X;\sigma,\delta]\). 
We get from \eqref{lr} and \eqref{primas} that \(\{X^i : i \in \N\}\) is a basis of \(A[X;\sigma,\delta]\) as right \(A\)--module, which implies, by \eqref{primas}, that \(A[X;\sigma,\delta] = [X;\sigma',\delta']A\). 

\medskip

 A definition of a  ring structure on \(A\s{X}\) extending that of \(A[X;\sigma,\delta]\)  might not be possible. We are interested in reasonable products according to Definition \ref{prodrazonable} below. Set, for any formal power series \(s = \sum_{i=0}^\infty s_i X^i\) with coefficients \(s_i \in A\), and any additive endomorphism \(L : A \to A\), 
\[
L(s) = \sum_{i=0}^\infty L(s_i)X^i. 
\] 
 \begin{definition}\label{prodrazonable}
We will say that a ring structure on \(A\s{X}\) is \emph{reasonable} if  \(A\) is a subring of \(A\s{X}\) in the obvious way, the left \(A\)--module of \(A\s{X}\) is given by restriction of scalars of the ring extension \(A \leq A\s{X}\) and the product of \(A\s{X}\) satisfies, for any formal power series \(s\), the equality
\begin{equation}\label{shift}
\left(\sum_{i=0}^\infty s_iX^i\right)X = \sum_{i=0}^\infty s_iX^{i+1}.
\end{equation} 
\end{definition}

\begin{lemma}\label{razonablea}
Given a reasonable ring structure on \(A\s{X}\) and \((\sigma,\delta)\) a left skew derivation of \(A\),  it turns out that \(A[X;\sigma,\delta]\) becomes, with its usual product, a subring of \(A\s{X}\) if and only if
\begin{equation}\label{Oreseria}
Xa = \sigma(a)X + \delta(a),
\end{equation}
for all \(a \in A\). 
\end{lemma}
\begin{proof}
 Obviously, \(A[X]\) is a left \(A\)--submodule of \(A\s{X}\). If \eqref{Oreseria} holds then, \(Xa \in A[X]\) for all \(a \in A\). Inductively,  by using \eqref{shift}, we get more generally that \(XaX^m \in A[X]\) for every \(m \geq 1\). We get thus that \(Xg \in A[X]\) for any \(g \in A[X]\). From this it is easily checked that \(fg \in A[X]\) for every \(f,g \in A[X]\). Therefore, \(A[X]\) is a subring of \(A\s{X}\) which, by \eqref{Oreseria}, is isomorphic to \(A[X;\sigma,\delta]\). The converse is clear. 
 \end{proof}

Lemma \ref{razonablea} makes reasonable the following definition.

\begin{definition}\label{SFPS}
Let \(A\s{X}\) be endowed with a reasonable product that satisfies \eqref{Oreseria} for some left skew derivation \((\sigma,\delta)\) of \(A\). We then say that this ring, denoted by \(A\s{X,\sigma,\delta}\), is the ring of left skew formal power series with respect to \((\sigma,\delta)\). Shortly, we say just that the ring \(A\s{X;\sigma,\delta}\) exists. 
\end{definition}

 It is not clear, in general, whether \(A\s{X;\sigma,\delta}\) exists.  
 Under the assumption that \(\sigma\) is an automorphism, in conjunction with suitable topological hypotheses on \(A\),  \(\sigma\) and \(\delta\), a reasonable ring structure is defined in \cite{Schneider/Venjakob:2005}.  When \(\delta\) is assumed to be locally nilpotent, that is, for any \(a\ \in A\), \(\delta^k(a)=0\) for \(k\) large enough,  there is a well defined skew product of  formal power series  extending that of \(A[X;\sigma,\delta]\), see \cite[Theorem 5.3]{Greenfeld/alt:2019}).  This last product requires no additional conditions. In the first case, the ring of skew Laurent formal power series, as a (left and right) ring of fractions of the skew power series ring is defined  in \cite{Schneider/Venjakob:2010}, added that \(\delta\) and \(\delta'\)  are nilpotent.

Let us describe the product defined in \cite{Greenfeld/alt:2019}, which subsumes that of \cite{Schneider/Venjakob:2010} . According to \cite[Lemma 5.2]{Greenfeld/alt:2019}, if \(\delta\) is locally nilpotent, then, for every \(a \in A\) and \(n \in \mathbb{N}\),  there exists a natural number \(O_n(a)\) such that, with the product of \(A[X;\sigma,\delta]\), 
\begin{equation}\label{banda} 
(X^ma)_i = 0 \text{ whenever } i \leq n \text{ and } m \geq O_n(a).
\end{equation} 
The skew product of \(s,t \in A\s{X}\) is then defined in \cite[Theorem 5.3]{Greenfeld/alt:2019} as follows. For every \(n \in \mathbb{N}\), set \(q \geq \max_{0 \leq i \leq n} \{O_n(t_i)\}\) and 
\begin{equation}\label{terminon}
(st)_n = \left[ \left (\sum_{i=0}^{q}s_iX^i\right)\left(\sum_{i=0}^nt_iX^i\right) \right]_n,  
\end{equation}
where the right hand product is done in \(A[X;\sigma,\delta]\). This gives the \(n\)--th coefficient of the product series \(st\), independently on the choice of \(q\), by virtue of  \eqref{banda}.  A straightforward argument shows that this product on \(A\s{X}\) is reasonable and \eqref{Oreseria} holds.

\begin{proposition}\label{productonil}
If \(\delta\) is locally nilpotent, then there exists a unique reasonable ring structure on \(A\s{X}\) containing \(A[X;\sigma,\delta]\) as a subring. 
The product obeys \eqref{terminon}  and, moreover,
\begin{equation}\label{Oreserie}
Xs = \sigma(s)X + \delta(s), 
\end{equation}
for every \(s \in A\s{X;\sigma,\delta}\).
\end{proposition}
\begin{proof}
By \cite[Theorem 5.3]{Greenfeld/alt:2019}, the product on \(A\s{X}\) given by \eqref{terminon} is well defined. All conditions required to endow \(A\s{X}\) with the structure of a ring are clearly satisfied by this product, with the possible exception of the associativity. Let us check it. To this end, given \(s \in A\s{X}\) and \( n \in \mathbb{N}\), set
\[
tr_n(s) = \sum_{i=0}^n s_iX^i \in A[X;\sigma,\delta], 
\]
and 
\[
q_n(s) = \max_{0\leq i \leq m \leq n}\{ O_m(s_i)\}.
\]

Given \(s, t \in A\s{X}\), we get from \eqref{banda} that, for any \(q \geq q_n(t)\),
\begin{gather*}
tr_n(st) = \sum_{m=0}^n (st)_mX^m = \sum_{m=0}^n\left[\left(\sum_{i=0}^{q} s_iX^i\right)\left(\sum_{i=0}^mt_iX^i\right)\right]_mX^m \\ = \sum_{m=0}^n\left[\left(\sum_{i=0}^{q} s_iX^i\right)\left(\sum_{i=0}^nt_iX^i\right)\right]_mX^m,  \end{gather*}
since 
\[
\left[\left(\sum_{i=0}^{q} s_iX^i\right)\left(\sum_{i=m+1}^nt_iX^i\right)\right]_m = 0.\]
Therefore, for \(q \geq q_{n}(t)\),
\begin{equation}
tr_n(st) = tr_n(tr_{q}(s)tr_n(t)). 
\end{equation}
Finally, if \(s,t,u \in A\s{X}\), then, for \(q \geq q_n(st)\), we have
\[
tr_n(u(st)) = tr_n(tr_q(u)tr_n(st)) = tr_n(tr_q(u)tr_{q_n(t)}(s)tr_n(t)),
\]
while, for \(q \geq q_{q_n(t)}(s)\), we get
\[
tr_n((us)t) = tr_n(tr_{q_n(t)}(us)tr_n(t)) = tr_n(tr_q(u)tr_{q_n(t)}(s)tr_n(t)).
\]
In this way, setting \(q \geq q_n(st), q_{q_n(t)}(s)\), we see that \(tr_n(u(st)) = tr_n((us)t)\) for every \(n\), which certainly implies that \(u(st) = (us)t\), as desired.

This product clearly satisfies \eqref{shift}. If \(\cdot\) is the product of some ring structure on \(A\s{X}\) extending that of \(A[X;\sigma,\delta]\) and satisfying \((\sum_{i=0}^\infty s_iX^i)\cdot X = \sum_{i=0}^\infty s_iX^{i+1}\) for every series \(s\), then, clearly, \((s\cdot X^m)_n = 0\) for every \(m > n\geq 0\).
Given \(s,t \in A\s{X}\), and \(n \in \mathbb{N}\), write
\[
s = \sum_{i=0}^{q_n(t)} s_iX^i + \sum_{i > q_n(t)}s_iX^i, \quad t = \sum_{i=0}^nt_iX^i + \sum_{i>n}t_iX^i. 
\]
Thus,
\begin{multline*}
(s\cdot t)_n = \left(s\cdot\sum_{i=0}^ns_iX^i\right)_n +\left(s \cdot \sum_{i > n}t_iX^i\right)_n  \\ = \left(s\cdot\sum_{i=0}^nt_iX^i\right)_n + \left(s \cdot \sum_{i > n}t_iX^{i-n-1}\cdot X^{n+1}\right)_n  = \left(s\cdot\sum_{i=0}^nt_iX^i\right)_n . \end{multline*}
Now, by using that \(\cdot\) extends de product of \(A[X;\sigma,\delta]\) and \eqref{terminon}, we get
\begin{multline*}
\left(s\cdot\sum_{i=0}^nt_iX^i\right)_n = \left(\left(\sum_{i=0}^{q_n(t)} s_i X^i\right)\cdot \left(\sum_{i = 0}^nt_iX^i\right)\right)_n + \left(\left(\sum_{i> q_n(t)}s_i X^i\right)\cdot \left(\sum_{i = 0}^nt_iX^i\right)\right)_n \\ = \left(\left(\sum_{i=0}^{q_n(t)} s_i X^i\right)\left(\sum_{i = 0}^nt_iX^i\right)\right)_n + \left(\left(\sum_{i> q_n(t)}s_i X^{i-q_n(t)-1}\right)\cdot \left(X^{q_{n}(t)+1}\sum_{i = 0}^nt_iX^i\right)\right)_n \\ = \left(\left(\sum_{i=0}^{q_n(t)} s_i X^i\right)\left(\sum_{i = 0}^nt_iX^i\right)\right)_n . 
\end{multline*}
Therefore, 
\[
(s\cdot t)_n = \left(\left(\sum_{i=0}^{q_n(t)} s_i X^i\right)\left(\sum_{i = 0}^nt_iX^i\right)\right)_n ,
\]
for all \(n\), which implies that \(s\cdot t = st\). 

Finally, \eqref{Oreserie} is easily deduced from \eqref{terminon} and \eqref{shift}. 
\end{proof}

\begin{definition}
Assume that the ring \(A\s{X;\sigma,\delta}\) exists according to Definition \ref{SFPS}.  We say that this ring structure is \emph{compatible} if \eqref{Oreserie} holds for every \(s \in A\s{X;\sigma,\delta}\). 
\end{definition}

The following theorem shows that there exists a tight connection between the existence of a reasonable product for skew power series and left fractions of skew polynomials with denominators at powers of \(X^{-1}\). We use the standard notion of left or right denominator set (see, e.g., \cite[Ch. II]{Stenstrom:1975}). 

\begin{theorem}\label{locnil}
The following conditions are equivalent for a left skew derivation \((\sigma,\delta)\) on a ring \(A\). 
\begin{enumerate}[(i)]
\item\label{ln} \(\delta\) is locally nilpotent; 
\item\label{locser} \(A\s{X}\) is endowed with a reasonable ring structure containing \(A[X;\sigma,\delta]\) as a subring and such that \(\Sigma = \{1,X,X^2,\dots \}\) is a left denominator set;
\item\label{locpol} \(\Sigma = \{1,X,X^2,\dots \}\) is a left denominator set of \(A[X;\sigma,\delta]\).
\end{enumerate}
The ring structure mentioned in \eqref{locser} is unique.  Moreover, it is compatible and its product satisfies \eqref{terminon}. 
\end{theorem}
\begin{proof}
\eqref{ln} \(\Rightarrow \) \eqref{locser}. Proposition \ref{productonil} gives the reasonable ring structure on \(A\s{X}\) containing \(A[X;\sigma,\delta]\) as a subring and also the uniqueness statement. Also, it is compatible and \eqref{terminon} holds. 

Let us check that \(\Sigma\) is a left denominator set. Since the product is reasonable, we easily get from \eqref{shift} that \(\Sigma\) is a left reversible multiplicative set in the sense of \cite[pag. 52]{Stenstrom:1975}. Given \(f = \sum_{i=0}^{\infty} f_iX^i \in A\s{X;\sigma,\delta}\), observe that \(f-f_0 = hX\), for a suitable \(h \in A\s{X;\sigma,\delta}\). Set \(n\) such that \(\delta^n(f_0) = 0\). Then
\[
X^nf = X^n(f-f_0) + Xf_0 = X^n h X + \sum_{k=0}^{n}N_k^n(f_0)X^k = X^n h X + (\sum_{k=1}^{n}N_k^n(f_0)X^{k-1})X, \] 
since \(\delta^n(f_0) = N_0^n(f_0)\). We see that \( X^nf = gX \) for some \(g \in A\s{X;\sigma,\delta}\). From this, we easily get inductively that \(\Sigma\) is a left permutable set and, therefore, a left denominator set.  

\eqref{locser} \(\Rightarrow\) \eqref{locpol}. If \(\Sigma\) is a left denominator for the reasonable ring structure on \(A\s{X}\) containing \(A[X;\sigma,\delta]\) as a subring, then, given \(f \in A[X;\sigma,\delta]\) and \(m \geq 0\), there are \(s \in A\s{X}\) and \(n \geq 0\) such that \(X^nf = sX^m\). Since \(X^nf \in A[X;\sigma,\delta]\), we have, by \eqref{shift}, that \(s \in A[X;\sigma,\delta]\). Thus, \(\Sigma\) is a left denominator set of \(A[X;\sigma,\delta]\).

\eqref{locpol} \(\Rightarrow\) \eqref{ln}. Assume \(\Sigma\) to be a left denominator set of \(A[X;\sigma,\delta]\).  Given  $a\in A$, there exist \(g \in A[X;\sigma,\delta]\) and \(n \in \mathbb{N}\) such that $X^na=gX$. By \eqref{Xnf}, 
\[gX=X^na=\sum_{i=0}^{n}N_{i}^{n}(a)X^i.\]
So $\delta^{n}(a)=N_{0}^{n}(a)=0$, as we desired. 
\end{proof}

\begin{corollary}
Assume \(\sigma\) to be an automorphism of \(A\). Then \(\Sigma\) is a right denominator set of \(A[X;\sigma,\delta]\) if, and only if, \(\delta' = - \delta \sigma^{-1}\) is locally nilpotent.
\end{corollary}
\begin{proof}
We have seen before that \(A[X;\sigma,\delta] = [X;\sigma',\delta']A\), where \(\sigma' = \sigma^{-1}\). The statement follows from the right handed version of the equivalence between \eqref{ln} and \eqref{locpol} in Theorem \ref{locnil}. 
\end{proof}

\begin{remark}\label{DerechanoesIzquierda}
Theorem \ref{locnil} shows that, if \(\delta'\) is locally nilpotent, then the right \(A\)--module of formal power series, which we may denote by \(\s{X}A\), is endowed with the structure of a ring containing \([X;\sigma',\delta']A\) as a subring. According to the version of Definition \ref{SFPS} for right skew derivations, this ring should be denoted by \(\s{X;\sigma',\delta'}A\). However, the rings \(A\s{X;\sigma,\delta}\) and \(\s{X;\sigma',\delta'}A\) might not be isomorphic in the expected way when \(\delta\) is locally nilpotent as well,  even though that \(A[X;\sigma,\delta] = [X;\sigma',\delta']A\).  In particular, \(\Sigma\) could be not a right denominator set for \(A\s{X;\sigma,\delta}\).  See Example \ref{polinomials} later. 
\end{remark}

We would like to explore the possibility of extending the ring structure from \(A\s{X;\sigma,\delta}\) given in Theorem \ref{locnil} to \(A\ls{X}\). Even though that the obvious candidate is the ring of left fractions \(\Sigma^{-1}A\s{X;\sigma,\delta}\),  these fractions could be not in general properly represented as elements of \(A\ls{X}\) (see Example \ref{polinomials} below). In order to solve this difficulty, we will investigate when \(\Sigma\) is a right denominator set of \(A\s{X;\sigma,\delta}\).

\begin{lemma} \label{permutablereversiblequatios} Let \((\sigma,\delta)\) be a left skew derivation on \(A\) such that \(A\s{X;\sigma,\delta}\) exits with a compatible product. The following statements hold:
\begin{enumerate}[(i)]
\item  \label{lemmapr1}  \(\Sigma\) is a right permutable set of \(f \in A\s{X;\sigma,\delta}\), if, and only if, for each  \(f \in A\s{X;\sigma,\delta}\), we have that 
\begin{equation}\label{right permutable equation}
fX^m= \sigma(s)X + \delta(s),
\end{equation}
for some \(s \in A\s{X;\sigma,\delta}\) and \(m\in \mathbb{N}\).

\item \label{lemmapr2}     \(\Sigma\) is a right  reversible set of \(A\s{X;\sigma,\delta}\) if, and only if, the only solution to \begin{equation}\label{right reversible equation}
0 = Xs
\end{equation}
is the trivial one. 
\end{enumerate}
\end{lemma}

\begin{proof}
(\ref{lemmapr1})  By \eqref{Oreserie} , we have that  \ref{right permutable equation} is equivalent to \(fX^m= Xs= \sigma(s)X+\delta(s)\). So, by induction, we have that   \(\Sigma\) is a right permutable set. 

(\ref{lemmapr2})  If \(\Sigma\) is right reversible then, from \(Xs = 0\) we should get that \(sX^k = 0\) for some \(k\). By \eqref{shift}, \(s = 0\). The converse is clear. 
\end{proof}

If \(f = \sum_{i=0}^{\infty} f_iX^i\) and \(s=  \sum_{i=0}^{\infty} s_iX^i   \), then  \eqref{right permutable equation}   is equivalent to 
\begin{equation}\label{right permutable equations}
\begin{array}{rl} 
0=& \delta(s_0) \\
0 = &\sigma(s_0) +  \delta(s_1) \\
\vdots \\
0 = &\sigma(s_{m-2}) +  \delta(s_{m-1}) \\
 f_0 = &\sigma(s_{m-1}) +  \delta(s_{m}) \\
 f_1 =  &\sigma(s_{m}) +  \delta(s_{m+1}) \\
\vdots \\
f_{i} = &\sigma(s_{m+i-1}) +  \delta(s_{m+i}) \\
\vdots \\
\end{array}
\end{equation}
and \eqref{right reversible equation} is equivalent to
\begin{equation}\label{right reversible equations}
\begin{array}{rl} 
0=& \delta(s_0) \\
0 = &\sigma(s_0) +  \delta(s_1) \\
\vdots \\
0= &\sigma(s_{i-1}) +  \delta(s_{i}) \\
\vdots \\
\end{array}
\end{equation}

\begin{proposition} \label{denominatordeltasurjective}
Let \((\sigma,\delta)\) be any left skew derivation on \(A\) such that  \(A\s{X;\sigma,\delta}\) exists with a compatible product. Set \(\Sigma = \{1,X,X^2, \dots \}\).  If \(\delta\) is surjective, the following statements hold:
\begin{enumerate}[(i)]
\item  \label{pr1}  \(\Sigma\) is a right permutable set of  \(A\s{X;\sigma,\delta}\).
%\item
\item \label{pr2}   \(\Sigma\) is not a right  reversible set of \(A\s{X;\sigma,\delta}\).

\end{enumerate}
\end{proposition}
\begin{proof}
 (\ref{pr1})  Let us prove that \(\Sigma\) is a right permutable set of \(A\s{X;\sigma,\delta}\).
 Given \(f = \sum_{i=0}^{\infty} f_iX^i \in A\s{X;\sigma,\delta}\), we need  to find \(s= \sum_{i=0}^{\infty} s_iX^i  \in A\s{X;\sigma,\delta}\) and \(m\in \mathbb{N}\), such that satisfy \eqref{right permutable equations},  by Lemma \ref{permutablereversiblequatios}.

Since  \(\delta\) is surjective, given \(f\), define \(s\in A\s{X;\sigma,\delta}\) recursively by \(\delta(s_0)=f_0 \) and  \(\delta(s_i)=f_i-\sigma(s_{i-1}) \) for \(i\geq 1\). This series \(s\) is a solution to \eqref{right permutable equations} for \(m = 0\).

 (\ref{pr2}) By the subjectivity of \(\delta\), we can define \(s\in A\s{X;\sigma,\delta}\) recursively by \(s_0=1\) and \(s_{i+1}\in A\) such that \(\delta(s_{i+1})=-\sigma(s_i)\). Then \(Xs=0\) and thus,   \(\Sigma\) is not a right reversible set of \(A\s{X;\sigma,\delta}\), by Lemma \ref{permutablereversiblequatios}.
 
\end{proof}

\begin{example} \label{polinomials}
Let  \(\F[Y]\) be the commutative polynomial ring over a field \(\F\) of characteristic \(0\). Take \(\sigma=id\), the identity map, and \(\delta\) the usual derivation. Then \(\delta\) is locally nilpotent and surjective and thus, by Theorem \ref{locnil}, \(A\s{X; \sigma, \delta}\) exists with a compatible product and  \(\Sigma\) is a  left denominator set of \(A\s{X; \sigma, \delta}\). By Proposition  \ref{denominatordeltasurjective},  \(\Sigma\) is a right permutable set but not a right reversible. Note that, if there exists a ring structure on \(A\ls{X}\) containing \(A\s{X;\sigma,\delta}\) as a subring and such that \(X^{-1}X = 1\), then \(Xs=0\) forces \(s = 0\). So, no such a ring structure does exist in this example.

Both \(A\s{X;\sigma,\delta}\) and \(\s{X;\sigma', \delta'}A\) do exist (see Remark \ref{DerechanoesIzquierda}) for this notation). However, there is no an isomorphism between them extending the identity map on \(A[X;\sigma,\delta] = [X;\sigma',\delta']A\):  Indeed, assume that  \(\alpha: A\s{X;\sigma,\delta} \to \s{X;\sigma', \delta'}A\) is any injective ring homomorphism such that \(\alpha(X) = X\). We can choose, as in the proof of Proposition \ref{denominatordeltasurjective},  \(0 \neq s \in A\s{X;\sigma,\delta}\) such that \(Xs= 0 \). Then \(0 = \alpha(Xs) = \alpha(X) \alpha(s) = X\alpha(s)\). This equality in \(\s{X;\sigma', \delta'}A\) implies that \(\alpha(s) = 0\), which is not possible. 
 \end{example}

Next, we will discuss when the Laurent series left \(A\)--module \(A\ls{X}\) becomes a ring containing \(A\s{X;\sigma,\delta}\) as a subring.

\begin{lemma}\label{LaurentL}
Let \(A\s{X}\) be endowed with a  ring structure such that \eqref{shift} holds. Set \(\Sigma = \{1,X,X^2, \dots \}\). The following statements are equivalent.
\begin{enumerate}[(i)]
\item\label{fr1} \(A\ls{X}\) has a ring structure such that \(X^{i}\) and \(X^{-i}\) are mutually inverse for any \(i\), and \(A\s{X}\) is a subring;
\item\label{fr2} \(A\s{X}\) has a ring structure such that \(\Sigma\) is a right denominator set. 
\end{enumerate}
In such a case,  \(A\ls{X}\)  is isomorphic to \(A\s{X}\Sigma^{-1}\) with a ring isomorphism whose restriction to \(A\s{X}\) is the canonical map  \(A\s{X} \to A\s{X}\Sigma^{-1}\). 
\end{lemma}
\begin{proof}
Assume \eqref{fr1}. Given any \(t = \sum_{i=i_0}^\infty t_i X^i \in A\ls{X}\), with \(i_0 < 0\), we use \eqref{shift} and that any \(X^i\) and \(X^{-i}\) are mutually inverse to do the following computation. 
\[
\begin{array}{rl}
t X^{-i_0} =&  \left(\sum_{i=i_0}^{-1}t_iX^i\right) X^{-i_0} + \left(\sum_{i= 0}^\infty t_iX^i\right)X^{-i_0} \\
=& \sum_{i=i_0}^{-1}t_iX^{i - i_0} + \sum_{i= 0}^\infty t_iX^{i-i_0} \\
= &\sum_{i=i_0}^\infty t_iX^{i-i_0}. 
\end{array}
\]

Then, every element of \(A\ls{X}\) is of the form \(sX^{\ell}\) for some \(s \in A\s{X}\) and \(\ell \in \mathbb{Z}\). Therefore, \(A\ls{X}\) is a right ring of fractions of \(A\s{X}\) with respect to \(\Sigma\). By \cite[II.1.4]{Stenstrom:1975}, \(\Sigma\) is a right denominator set of \(A\s{X}\). 

Conversely, asume \eqref{fr2}, which implies that there exists a ring of fractions \(A\s{X}\Sigma^{-1}\). The kernel of its canonical map \(A\s{X} \to A\s{X}\Sigma^{-1}\) is
\[
 \{s \in A\s{X} : sX^n =0 \text{ for some } n \geq 0\},
 \] 
 which is zero by \eqref{shift}.  
 
 Every nonzero element of  \(A\s{X}\Sigma^{-1}\) is of the form \((\sum_{i = j_0}^\infty s_iX^i)X^{k_0}\), with \(j_0 \geq 0, k_0\leq 0\) and \(s_{j_0} \neq 0\).  Thus, by \eqref{shift}, 
\[
\left(\sum_{j = j_0}^\infty s_jX^j\right)X^{k_0} = \left(\left(\sum_{j=0}^\infty s_{j + j_0}X^j\right)X^{j_0}\right)X^{k_0} = \left(\sum_{j=0}^\infty s_{j + j_0}X^j\right)X^{j_0 + k_0}.
\]
Thus, every nonzero element of  \(A\s{X}\Sigma^{-1}\) is  written as \(s X^{i_0}\) for \(s \in A\s{X}\) with \(s_0 \neq 0\) and \(i_0 \in \mathbb{Z}\). Let us check that this expression is unique: suppose that \((\sum_{i=0}^\infty s_iX^i)X^{i_0} = (\sum_{j=0}^\infty t_jX^i)X^{j_0} \) with \(s_0\neq 0, t_0 \neq 0\) and \(i_0,j_0 \in \mathbb{Z}\). We may assume that \(i_0 \geq j_0\). Hence,  \((\sum_{i=0}^\infty s_iX^i)X^{i_0-j_0} = \sum_{j=0}^\infty t_jX^i \) in \(A\s{X}\Sigma^{-1}\). Since the canonical map is injective, the former equality also holds in \(A\s{X}\).  By \eqref{shift} this is possible only if \(i_0 = j_0\) and  \(\sum_{i=0}^\infty s_iX^i = \sum_{j=0}^\infty t_jX^i \). 

The map sending \((\sum_{i=0}^\infty s_iX^i)X^{i_0}\) onto \(\sum_{i=i_0}^\infty s_iX^{i+i_0}\) gives an isomorphism of left \(A\)--modules from \(A\s{X}\Sigma^{-1}\) onto \(A\ls{X}\). We may use it to transfer the ring structure of \(A\s{X}\Sigma^{-1}\) to \(A\ls{X}\), which satisfies \eqref{fr1}. Finally, the restriction of the inverse of this map to \(A\s{X}\) is clearly the canonical map from \(A\s{X}\) to \(A\s{X}\Sigma^{-1}\).
 \end{proof}

\begin{definition}\label{SkewLaurentSeries}
Assume that \(A\s{X;\sigma,\delta}\) exists with a compatible product. If \(\Sigma = \{1,X,X^2, \dots \}\) is a right denominator set  then the ring of \emph{left skew Laurent formal power series} is defined as the ring of fractions \(A\ls{X;\sigma,\delta} = A\s{X;\sigma,\delta}\Sigma^{-1}\). According to Lemma \ref{LaurentL}, this ring structure can be understood to be defined on \(A\ls{X}\). This ring contains \(A\s{X;\sigma,\delta}\) and \(A[X;\sigma,\delta]\) as subrings.
\end{definition}

Assume  \(A\s{X;\sigma,\delta}\) to be endowed with a reasonable and compatible ring structure for \((\sigma,\delta)\).  The identity \eqref{Oreserie} extends to
\begin{equation}\label{Xns}
X^ns = \sum_{k=0}^n N_k^{n}(s)X^k, \quad (s \in A\s{X;\sigma,\delta}). 
\end{equation}
for additive endomorphisms of \(A\s{X;\sigma,\delta}\) defined recursively, for \(0 \leq i \leq n\), by
\[
N_{i}^{n+1} = \sigma N_{i-1}^n + \delta N_i^n, \qquad (N_{-1}^n = 0, N_{0}^0 = id_{A\s{X;\sigma,\delta}}, N_{n+1}^n = 0). 
\]

When \(\sigma\) is an automorphism, observe that, for all \(s \in A\s{X;\sigma,\delta}\), one deduces from \eqref{Oreserie} that
\begin{equation}
sX = X\sigma'(s) + \delta'(s),
\end{equation}
and, by induction,
\begin{equation}\label{sXn}
sX^n = X\sum_{k=0}^{n-1}\sigma'\delta'^k(s)X^{n-1-k} + \delta'^n(s), 
\end{equation}
for all \(n \geq 1\).

\begin{proposition}\label{LaurentSeries}
Let \((\sigma,\delta)\) be a left skew derivation on \(A\) such that \(A\s{X;\sigma,\delta}\) is endowed with a reasonable and compatible ring structure. If \(\sigma\) is an automorphism, then \(\Sigma = \{1,X,X^2, \dots\}\) is a right permutable set of \(A\s{X;\sigma,\delta}\) if, and only if, for every \(s \in A\s{X;\sigma,\delta}\) there exist \(n\in \mathbb{N}\) and \(t \in A\s{X;\sigma,\delta}\) such that \(\delta'^n(s) = Xt\). 
\end{proposition}
\begin{proof}
If \(\Sigma\) is a right permutable set, then, given \(s \in A\s{X;\sigma,\delta}\), there are \(n \in \N, r \in A\s{X;\sigma,\delta}\) such that \(sX^n = Xr\). From \eqref{sXn} we get 
\[
\delta'^n(s) = X\left(r - \sum_{k=0}^{n-1}\sigma'\delta'^k(s)X^{n-1-k}\right).
\]
Conversely, if, given \(s \in A\s{X;\sigma,\delta}\), we have that  \(\delta'^n(s) = Xt\) for some \(t \in A\s{X;\sigma,\delta}\), then, by \eqref{sXn}, 
\[
sX^n = X\left(\sum_{k=0}^{n-1}\sigma'\delta'^k(s)X^{n-1-k} + t\right). 
\]
An inductive argument shows then that \(\Sigma\) is a right permutable set for \(A\s{X;\sigma,\delta}\). 

\end{proof}

\begin{theorem}  \label{nilLaurent} Let \((\sigma,\delta)\) be a left skew derivation on \(A\) such that \(A\s{X;\sigma,\delta}\) exits with a compatible ring structure.  If \(\sigma\) is an automorphism and \(\delta^{'}\) is nilpotent, then  \(\Sigma\) is a right denominator set in \(A\s{X;\sigma,\delta}\)  and the product of \(A\ls{X;\sigma,\delta}\) satisfies that
\begin{equation}\label{Xmenosunos}
X^{-1}s = \sum_{k=0}^m\sigma'\delta'^{k}(s)X^{-1-k}, \qquad (s \in A\s{X;\sigma,\delta}). 
\end{equation}
\end{theorem}

\begin{proof}
 \(\Sigma\)  is right permutable, by Proposition \ref{LaurentSeries}.
Now, suppose that there exists \(s \in A\s{X;\sigma,\delta} \) such that
\(0= Xs\). Define  \(t = \sum_{i= 0}^\infty t_iX^i \in A\s{X;\sigma,\delta} \) with \(t_i = - \sigma(s_i)\).  Then, for every \(i,j\) we deduce, by using \eqref{right reversible equations}, that \(\delta^{'i}(t_{i+j}) =  t_j\). Since  \(\delta^{'}\) is nilpotent, we obtain that \(t=0\) which implies that \(s = 0\) because \(\sigma\) is an automorphism.  By Lemma \ref{permutablereversiblequatios}, \(\Sigma\) is right reversible. 

Finally, \eqref{Xmenosunos}  is immediate from  \eqref{sXn}
\end{proof}

Next example  shows that the local nilpotency of \(\delta\) and \(\delta'\), together with the right reversibility of \(\Sigma\), is not sufficient to ensure that \(\Sigma\) is a right denominator set.

\begin{example} Let \(A=\F[Y,Z]\) where  \(\F\) is a field and  \(\F[Y,Z]\) is the commutative polynomial ring with coefficients in \(\F\) over the variables \(Y\)  and \(Z\).   Take \(\sigma=id\) and define \(\delta(f(Y,Z))=Y \frac{\partial f}{\partial Z}\). It is clear that \(\delta\) and \(\delta'=-\delta\) are both derivations which are locally nilpotent but not nilpotent. We claim that there is no \(t = \sum_{i= 0}^\infty t_iX^i \in A\s{X;\sigma,\delta} \) such that \(\delta^{'i}(t_{i+j}) =  t_j\) and \(\delta'(t_0)=0\). Indeed, if  such a \(t\) exists, then \(t_0=f(Y)\) for some polynomial in \(\F[Y]\). Now, \(t_0= \delta'(t_1)=-Y\frac{\partial f(Z)}{\partial Z}\), and thus \(Y\) divides \(t_0\). Inductively, \(Y^k\) divides \(t_0\) for all \(k\geq 1\), which is not possible unless \(t=0\). So, the only solution to \eqref{right reversible equations} is the trivial one. Then, by Lemma \ref{permutablereversiblequatios}, \(\Sigma\) is right reversible set.

Let us see that   \(\Sigma\) is not a right permutable set of \(A\s{X;\sigma,\delta}\). Suppose that  \(\Sigma\) is a right permutable set, then, for  \(f= \sum_{i= 0}^\infty Z^{i+1}X^i \in A\s{X;\sigma,\delta} \), there are  \(s= \sum_{i= 0}^\infty s_iX^i \in A\s{X;\sigma,\delta} \) and \(k \geq0\) such that \(fX^{k}= Xs\). Then by \eqref{right permutable equations},  we have that
\[Z = s_{k-1}+\delta(s_k) \text{ and } Z^2 = s_{k}+\delta(s_{k+1}). \]

Since \(\delta(s_k) = Y\frac{\partial s_k}{\partial Z} \), we have that \(s_{k-1}=Z\) and \(s_k\in \F[Y]\). But this makes impossible that \(Z^2 = s_{k}+Y\frac{\partial s_{k+1}}{\partial Z}\). So,  \(\Sigma\) is not a right permutable set of \(A\s{X;\sigma,\delta}\).

\end{example}

We collect in the following theorem some  information relevant for our purposes concerning cyclic convolutional codes.

\begin{theorem}\label{LaurentSeries2}
Let \((\sigma,\delta)\) be a left skew derivation such that \(\sigma\) is an automorphism and \(\delta\) is locally nilpotent. Assume that  there exists \(m \geq 1\) such that \(\delta'^m = 0\). There is a unique ring structure on \(A\ls{X}\), denoted by \(A\ls{X;\sigma,\delta}\), containing \(A[X;\sigma,\delta]\) as a subring and such that the multiplication obeys the rules
\begin{equation}\label{shifting}
sX^\ell = \sum_{i=0}^\infty s_iX^{i+\ell},
\end{equation}
\begin{equation}\label{productolocal}
st = \sum_{n=0}^\infty \left[\left(\sum_{i=0}^{q_n}s_iX^i\right)\left(\sum_{i=0}^nt_iX^i\right)\right]_n X^n,
\end{equation}
\begin{equation}\label{Xmenosunoconmutacion}
X^{-1}s = \sum_{k=0}^m\sigma'\delta'^{k}(s)X^{-1-k}, 
\end{equation}
for all \(\ell \in \mathbb{Z}, s, t \in A\s{X;\sigma,\delta}\) and \(q_n  \geq \max_{0 \leq i \leq n } \{O_n(t_{i})\}\) for all \(n \in \N\).
Moreover, \(A\s{X;\sigma,\delta}\) is a subring of \(A\ls{X;\sigma,\delta}\). 
\end{theorem}
\begin{proof}
 Proposition \ref{productonil} guarantees that \(A\s{X;\sigma,\delta}\) exists with a  compatible ring structure.  By Theorem \ref{nilLaurent}, we get the ring \(A\ls{X;\sigma,\delta}\) which, as a left module, is \(A\ls{X}\) (see Lemma \ref{LaurentL} and Definition \ref{SkewLaurentSeries}).  Moreover, it contains \(A\s{X;\sigma,\delta}\) as a subring. The product of this subring is unique by Proposition \ref{productonil} and determines uniquely that of the ring of fractions \(A\ls{X;
 \sigma,\delta} = A\s{X;\sigma,\delta}\Sigma^{-1}\).  Thus, the product of any two Laurent power series is determined by \eqref{Xmenosunoconmutacion}, \eqref{productolocal} and \eqref{shifting}. 
\end{proof}

\begin{remark}
The skew Laurent series ring from \cite[Proposition 1.2]{Schneider/Venjakob:2010} is a particular case of \ref{LaurentSeries2}. Observe that our construction does not require \(A\) to be Noetherian nor any topological condition on \(A\), \(\sigma\) or \(\delta\). 
\end{remark}

Finally, under the conditions of Theorem \ref{LaurentSeries2}, we find the skew Laurent polynomials form a subring of \(A\ls{X;\sigma,\delta}\), as the following proposition proves.

\begin{proposition}\label{x-n}
Suppose that  \(\sigma\) is an automorphism,  \(\delta\) is locally nilpotent and \(\delta'\)  is nilpotent, say of degree $m$, and let $s\in A\s{X; \sigma, \delta}$.  Then 

\begin{equation}\label{x -m}
X^{-n}s=\sum_{k=0}^{n(m-1)}\left(\sum_{\underset{k_j\leq m-1}{k_1+\cdots+k_n=k}}\sigma'\delta'^{k_1}\cdots\sigma'\delta'^{k_n}(s)\right)X^{-n-k}
\end{equation}
for each \(n\geq 1\). 

In particular, the Laurent polynomial left \(A\)--submodule \(A[X,X^{-1}]\) becomes a subring of \(A\ls{X;\sigma,\delta}\) which will be denoted by \(A[X,X^{-1}; \sigma,\delta]\). 
\end{proposition}

\begin{proof}
We proceed by induction. If \(n=1\), \eqref{x -m} is exactly \eqref{Xmenosunoconmutacion}. Now suppose that \eqref{x -m} holds for \(n\). Then, by  \eqref{Xmenosunoconmutacion} and inductive hypothesis,  
\[
\begin{array}{l}
X^{-n-1}s\\
=X^{-1} \left(\underset{k=0}{\overset{n(m-1)}{\sum}}\left(\underset{\underset{k_j\leq m-1}{k_1+\cdots+k_n=k}}{\sum}\sigma'\delta'^{k_1}\cdots\sigma'\delta'^{k_n}(s)\right)X^{-n-k}\right)\\
=\left(\underset{k=0}{\overset{n(m-1)}{\sum}}\left(\underset{\underset{k_j\leq m-1}{k_1+\cdots+k_n=k}}{\sum}X^{-1} \sigma'\delta'^{k_1}\cdots\sigma'\delta'^{k_n}(s)\right)X^{-n-k}\right)\\
=\left(\underset{k=0}{\overset{n(m-1)}{\sum}}\underset{\underset{k_j\leq m-1}{k_1+\cdots+k_n=k}}{\sum}\left(\underset{i=0}{\overset{m-1}\sum}\sigma'\delta^{i}( \sigma'\delta'^{k_1}\cdots\sigma'\delta'^{k_n}(s)\right)X^{-1-i}X^{-n-k}\right)\\
=\left(\underset{k=0}{\overset{(n+1)(m-1)}{\sum}}\left(\underset{\underset{k_j\leq m-1}{k_1+\cdots+k_{n+1}=k}}{\sum}\left(\sigma'\delta^{k_1}( \sigma'\delta'^{k_2}\cdots\sigma'\delta'^{k_{n+1}}(s))\right)\right)X^{-n-1-k}\right),\\
\end{array}
\]
as we required.
\end{proof}

Suppose that  \(\sigma\) is an automorphism and  that both \(\delta\) and \(\delta'\) are nilpotent. Choose \(m \geq 1\) such that \(\delta^m = \delta'^m = 0\). Using \eqref{terminon} and \eqref{ga}, and taking that \(q_n(a) = (n+1)m-1\) for any \(a \in A\) and all \(n \geq 0\) into account, we get

 \begin{equation}\label{lr infty}
\left(\sum_{i=0}^{\infty} s_iX^i\right)a = \sum_{i=0}^{\infty} \left(\sum_{j=i}^{(i+1)m-1}s_jN_{i}^j(a)\right)X^i, 
\end{equation}
for all \(a_i\in A\). 

\section{Modules}\label{modulos}

Given a right \(A\)--module \(M\) and a ring extension \(B\) of \(A\), the additive group \(M \otimes_A B\) is endowed with a right \(B\)--module structure determined by the rule \[(m \otimes_A b)b' = m \otimes_A bb'\] for any \(m \in M, b,b' \in B.\) 
This construction may be applied to \(B = A[X;\sigma,\delta]\).  There is a canonical isomorphism of additive groups
\[
M \otimes_A A[X;\sigma,\delta] \to M[X], 
\]
defined on generators of the tensor product by 
\[
m \otimes_A \sum_{i = 0}^{n} s_iX^i  \mapsto \sum_{i = 0}^{n}ms_iX^i. 
\]
The right \(A[X;\sigma,\delta]\)--module structure of \(M \otimes_A A[X;\sigma,\delta]\) is transferred to \(M[X]\).
Explicitly, by \eqref{Xnf}, this right module structure on \(M[X]\) is given by
\begin{equation}\label{AModuloDerecha}
\left(\sum_{i=0}^nm_iX^i\right)\left(\sum_{j=0}^\ell f_jX^j\right) = \sum_{j=0}^\ell\sum_{i=0}^n \left(\sum_{k=i}^n m_kN_i^k(f_j)\right)X^{i+j}.  
\end{equation}

Assume \(A\s{X;\sigma,\delta}\) exists with a compatible ring structure. Let
\[
\pi : M \otimes_A A\s{X;\sigma,\delta}  \to M\s{X}, 
\]
defined on generators of the tensor product by 
\[
m \otimes_A \sum_{i = 0}^{\infty} s_iX^i  \mapsto \sum_{i = 0}^{\infty}ms_iX^i. 
\]

If \(M\) is finitely presented as a right \(A\)--module, then the former map is an isomorphism (see, e.g. \cite[Lemma I.13.2]{Stenstrom:1975}). Thus, the right \(A\s{X;\sigma,\delta}\)--module structure on \(M \otimes_A A\s{X;\sigma,\delta}\) is transferred to a right \(A\s{X;\sigma,\delta}\)--module structure on \(M\s{X}\). It is useful to know how to compute \(\pi^{-1}\). Given a set of generators \(\{m_1,\dots,m_h\}\) of the right \(A\)--module \(M\), and \(s = \sum_{i = 0}^{\infty} s_iX^i \in M\s{X}\), set, for every \(i\), \(s_i = \sum_{k=1}^h m_ks_{ki}\), for adequate \(s_{ki} \in A\). Then 
\begin{equation}\label{pimenosuno}
\pi^{-1}\left(\sum_{i = 0}^\infty s_iX^i\right) = \sum_{k=1}^h m_k \otimes_A \sum_{i = 0}^\infty s_{ki} X^i . 
\end{equation}

As a consequence of \eqref{pimenosuno} and \eqref{shift}, the right \(A\s{X;\sigma,\delta}\)--module structure on \(M\s{X}\) satisfies the rule
\begin{equation}\label{shiftM}
\left(\sum_{i=0}^\infty s_iX^i\right)X^n = \sum_{i=0}^\infty s_iX^{i+n}.
\end{equation}

When \(\delta\) is locally nilpotent, the right \(A\s{X;\sigma,\delta}\)--structure of \(M\s{X}\) obeys a similar rule than the product of the skew formal series ring.

\begin{proposition}\label{productonilM}
Assume that \(\delta\) is locally finite. For any finitely presented right \(A\)--module \(M\), there is a right \(A\s{X;\sigma,\delta}\)--module structure on \(M\s{X}\) such that, by restriction of scalars, makes \(M[X]\) an \(A[X;\sigma,\delta]\)--submodule of \(M\s{X}\). Moreover, for \(s = \sum_{i=0}^\infty s_i X^i \in M\s{X}\) and \(t = \sum_{i=0}^\infty t_iX^i \in A\s{X;\sigma,\delta}\), for any sequence of non-negative integers \(q_n\) such that \(q_n  \geq \max_{0 \leq i \leq n } \{O_n(t_{i})\}\) we have
\begin{equation}\label{productos}
st = \sum_{n=0}^\infty \left[\left(\sum_{i=0}^{q_n}s_iX^i\right)\left(\sum_{i=0}^nt_iX^i\right)\right]_n X^n.
\end{equation}
\end{proposition}
\begin{proof}
We make use of formula \eqref{pimenosuno}. Thus, given \(t = \sum_{i= 0}^\infty t_iX^i \in A\s{X;\sigma,\delta}\), if we set, for any \(n \in \mathbb{N}\), 
\(
q_n \geq \max_{0 \leq i \leq n } \{O_n(t_{i})\},
\)
then we get
\begin{multline*}
(st)_n =  \sum_{k=1}^h m_k \left[ \left(\sum_{i =0}^{q_n}s_{ki}X^i\right)\left(\sum_{i=0}^{n}t_{i}X^i\right)\right]_n \\
= \left[ \left(\sum_{k=1}^h \sum_{i =0}^{q_n}m_ks_{ki}X^i\right)\left(\sum_{i=0}^{n}t_{i}X^i\right)\right]_n = \left[ \left( \sum_{i =0}^{q_n}s_{i}X^i\right)\left(\sum_{i=0}^{n}t_{i}X^i\right)\right]_n,
\end{multline*}
according to \eqref{terminon}.
\end{proof}

Proposition \ref{LaurentLM} implies that, under its hypotheses, we can consider a well defined right \(A\ls{X;\sigma,\delta}\)--module structure on \(M\ls{X}\).

\begin{proposition}\label{LaurentLM}
Assume that \(A\s{X;\sigma,\delta}\) exists with a  compatible ring structure for \((\sigma,\delta)\). If \(\Sigma = \{1, X, X^2,\dots \}\) is a right denominator set for \(A\s{X;\sigma,\delta}\), then, for every finitely presented right \(A\)--module \(M\), the canonical map \(M\s{X;\sigma,\delta} \to M\s{X;\sigma,\delta}\Sigma^{-1}\) is injective.  Moreover, every nonzero element of  \(M\s{X;\sigma,\delta}\Sigma^{-1}\) is uniquely written as \(s X^{i_0}\) for \(s \in M\s{X;\sigma,\delta}\) with \(s_0 \neq 0\) and \(i_0 \in \mathbb{Z}\). 
\end{proposition}
\begin{proof}
The kernel of the canonical map \(M\s{X} \to M\s{X}\Sigma^{-1}\) consists of those power series \(s\) such that \(sX^n = 0\) for some \(n \in \N\). It follows from \eqref{shiftM} that this kernel is zero. Following the steps of the proof of Proposition \ref{LaurentL} one obtains a valid one for the present proposition. 
\end{proof}

We finally record the version of Theorem \ref{LaurentSeries2} for finitely presented right \(A\)--modules.

\begin{theorem}\label{MSmodulo}
Let \((\sigma,\delta)\) be a left skew derivation such that \(\sigma\) is an automorphism and \(\delta\) is locally nilpotent. Assume that  there exists \(m \geq 1\) such that \(\delta'^m = 0\).  For every finitely presented right \(A\)--module \(M\), the Laurent power series module \(M\ls{X}\) is endowed with the structure of a right \(A\ls{X;\sigma,\delta}\)--module such that, by restriction of scalars, makes \(M\s{X}\) an \(A\s{X;\sigma,\delta}\)--submodule of \(M\ls{X}\). This module structure is determined by the rules
\[
sX^\ell = \sum_{i=0}^\infty s_iX^{i+\ell},
\]
\[
st = \sum_{n=0}^\infty \left[\left(\sum_{i=0}^{q_n}s_iX^i\right)\left(\sum_{i=0}^nt_iX^i\right)\right]_n X^n,
\]
\[
yX^{-1} t = \sum_{k=0}^my\sigma'\delta'^{k}(t)X^{-1-k},
\]
for all \(y \in M, \ell \in \mathbb{Z}, s \in M\s{X;\sigma,\delta}, t \in A\s{X;\sigma,\delta}\) and \(q_n  \geq \max_{0 \leq i \leq n } \{O_n(t_{i})\}\) for all \(n \in \N\).
\end{theorem}
\begin{proof}
It follows from Proposition \ref{productonilM}, Proposition \ref{LaurentLM} and Corollary \ref{nilLaurent}.
\end{proof}

\section{General cyclic convolutional codes}\label{biyeccion}

Let us first recall the classical definition of convolutional code according to C. Roos. By \(\F\) we denote any field. We have then the Laurent formal power series field \(\F\ls{X}\).  For any \(\F\)--vector space \(V\), the \(\F\)--vector space \(V\ls{X}\) becomes an \(\F\ls{X}\)--vector space according to the rule 
\begin{equation}\label{laesperada}
\left(\sum_{i}s_iX^i\right)\left(\sum_{j}t_jX^j\right) = \sum_k\left(\sum_{i+j = k}s_it_j\right)X^k,
\end{equation}
for all \(\sum_{i}s_iX^i \in V\ls{X}\) and \(\sum_{j}t_jX^j\in \F\ls{X}\). Analogously, \(V[X]\) becomes an \(\F[X]\)--module. 

\begin{definition}\cite[Definition 1]{Roos:1979}\label{CCRoos}
A convolutional linear code of rate $k/n$ is a $k$--dimensional vector subspace of $\ff^n\ls{X}$ with a basis consisting polynomial sequences, that is, of polynomial vectors in $\ff^n[X]$. 
\end{definition}

Our aim is to state a mathematically sound notion of cyclic convolutional code with respect to a left skew derivation \((\sigma,\delta)\) of a finite dimensional algebra \(A\) in the spirit of Definition \ref{CCRoos}. To this end, we will use the Laurent formal power series rings \(A\ls{X;\sigma,\delta}\) and their modules discussed in the previous sections. 

In what follows, \(A\) will denote a finite-dimensional algebra over the field \(\F\). Under this hypothesis, every locally nilpotent \(\F\)--linear endomorphism of  \(A\) becomes nilpotent. Along this section we will assume that \(\sigma\) is an \(\F\)--algebra automorphism of \(A\) and that both \(\delta\) and \(\delta'\) are nilpotent \(\F\)--linear skew derivations. We thus fix an index \(m\) such that \(\delta^m = \delta'^m = 0\). 

 Observe that the existence of this last \(\F\)--algebra structure on \(A\ls{X}\) is guaranteed by Theorem \ref{LaurentSeries2}. Given a finite-dimensional right \(A\)--module \(M\), the Laurent power series right \(A\)--module \(M\ls{X}\) is endowed with the structure of a right \(A\ls{X;\sigma,\delta}\)--module 
by virtue of Theorem \ref{MSmodulo}. In this way, by restriction of scalars, \(M\ls{X}\) is a right \(A\)--module. Precisely, it follows from \eqref{x -m} and \eqref{lr infty}, that, for \(a\in A\), \(s \in M\ls{X}\). 
\medskip
\begin{equation}\label{sx-ma}
\begin{array}{l}
sX^{-n_0}a \\
= s\overset{n_0(m-1)}{\underset{k=0}{\sum}}\left(\underset{k_1+\cdots+k_{n_0}=k}{\sum}\sigma'\delta'^{k_1}\cdots\sigma'\delta'^{k_{n_0}}(a)\right)X^{-n_0-k} \\
= \overset{n_0(m-1)}{\underset{k=0}{\sum}} \left(\underset{k_1+\cdots+k_{n_0}=k}{\sum}\left(\underset{i=0}{\overset{\infty}{\sum}}s_iX^{i}\right)\sigma'\delta'^{k_1}\cdots\sigma'\delta'^{k_{n_0}}(a)\right)X^{-n_0-k} \\
= \overset{n_0(m-1)}{\underset{k=0}{\sum}}\left(\underset{k_1+\cdots+k_{n_0}=k}{\sum} \,\,\underset{i=0}{\overset{\infty}{\sum}}\left(\overset{(i+1)m-1}{\underset{j=i}{\sum}}s_jN_{i}^j(\sigma'\delta'^{k_1}\cdots\sigma'\delta'^{k_{n_0}}(a))\right)X^i\right)X^{-n_0-k} \\
= \underset{i=0}{\overset{\infty}{\sum}} \left(\overset{(i+1)m-1}{\underset{j=i}{\sum}}\,\,\overset{n_0(m-1)}{\underset{k=0}{\sum}} \underset{k_1+\cdots+k_{n_0}=k}{\sum}s_jN_{i}^j(\sigma'\delta'^{k_1}\cdots\sigma'\delta'^{k_{n_0}}(a))\right)X^{i-n_0-k}. \\
\end{array}
\end{equation}
Also, since \(\F\ls{X}\) is a subring of \(A\ls{X;\sigma,\delta}\), we see that \(M\ls{X}\) becomes a right \(\F\ls{X}\)--vector space. By \eqref{productos} and \eqref{x -m}, this vector space structure is governed, as we might expect, by the rule \eqref{laesperada}, since \(\delta, \sigma\) and \(\delta'\) are \(\F\)--linear. 

\medskip

Fix an \(\F\)--basis \(\{a_1,\dots,a_r\}\) of \(A\).

\begin{lemma}\label{DSmodulo}
A subset \(\mathcal{D}\) of \(M\ls{X}\) is a right \(A\ls{X;\sigma,\delta}\)--submodule of \(M\ls{X}\) if, an only if, it is both an \(\F\ls{X}\)--vector subspace and a right \(A\)--submodule of \(M\ls{X}\). 
\end{lemma}
\begin{proof}
Let us just check that, if \(\mathcal{D}\) is an \(\F\ls{X}\)--vector subspace and a right \(A\)--submodule of \(M\ls{X}\), then it is a right \(A\ls{X;\sigma,\delta}\)--module. Thus, pick \(sX^{-n_0} \in \mathcal{D}\), where \(s \in M\s{X}\) and \(n \geq 0\), and \(t = \sum_it_iX^i  \in A\ls{X;\sigma,\delta}\). For each \(i\) write \(t_i = \sum_{j=1}^r t_{ij}a_j\) for suitable \(t_{ij} \in \F\). Thus, \(t = \sum_{j=1}^r a_j\sum_{i}t_{ij}X^i\). 
\begin{multline*}
sX^{-n_0} t = sX^{-n_0}\sum_{j=1}^r a_j \sum_{i}t_{ij}X^i \\ = s \sum_{j=1}^r \sum_{k=0}^{n_0(m-1)}\left(\sum_{k_1+\cdots+k_{n_0}=k}\sigma'\delta'^{k_1}\cdots\sigma'\delta'^{k_{n_0}}(a_j)\right)X^{-n_{0}-k}\sum_{i}t_{ij}X^i \\ = s \sum_{j=1}^r \sum_{k=0}^{n_{0}(m-1)}\left(\sum_{k_1+\cdots+k_{n_0}=k}\sigma'\delta'^{k_1}\cdots\sigma'\delta'^{k_{n_0}}(a_j)\right)\sum_{i}t_{ij}X^{i-n_{0}-k} \in \mathcal{D}, \end{multline*}
since \(s \in \mathcal{D}\). 
\end{proof}

We are in position of give a definition of general cyclic structure on a convolutional code extending that of \cite{Gomez/Sanchez:2024}. We thus assume \(\F^n\) to be endowed of a right \(A\)--module structure. 

\begin{definition}\label{CCC}
A convolutional code \(\mathcal{D}\) is said to be \((\sigma,\delta,A)\)--cyclic if \(\mathcal{D}\) is a right \(A\)--submodule of
\(\F^n\ls{X}\). Equivalently, by Lemma \ref{DSmodulo}, a \((\sigma,\delta,A)\)--cyclic convolutional code is just a right \(A\ls{X;\sigma,\delta}\)--submodule of \(\F^n\ls{X}\) having a set of generators in \(\F^n[X]\). 
\end{definition}

\begin{remark}
Assume \(\delta = 0\). Then the right \(A\)--module structure on \(\F^n\ls{X}\) given in \eqref{sx-ma} boils down to
\[
\left(\sum_{i \geq -n_0}^\infty s_iX^i\right)a =  \sum_{i \geq -n_0}^\infty s_i\sigma^i(a)X^i,
\]
which agrees with \cite[(3.1)]{Gomez/Sanchez:2024}. As a consequence, \((\sigma,0,A)\)--cyclic convolutional codes are just the \((\sigma,A)\)--CCCs discussed in \cite{Gomez/Sanchez:2024}. 
\end{remark}

When \(\delta = 0\), the map \(\mathcal{D} \mapsto \mathcal{D}\cap \F^n[X]\) is a bijection between the set of all \((\sigma,A)\)--CCCs codes of rate \(k/n\) and the set of all \(A[X;\sigma]\)--submodules of \(\F^n[X]\) that are \(\F[X]\)--direct summands of rank \(k\) (see \cite[Proposition 2]{Gomez/Sanchez:2024}). To extend this correspondence to the case of \((\sigma,\delta,A)\)--CCCs, we need to discuss some basic facts on convolutional codes, according to \cite{Gomez/Sanchez:2024}.

We need several lemmas. Given a subset \(B \subseteq \F^n\ls{X}\), the notation \(B^*\) stands for the \(\F\ls{X}\)--vector subspace of \(\F^{n}\ls{X}\) spanned by \(B\). 

\begin{lemma}\emph{(\cite[Lemma 1]{Gomez/Sanchez:2024})}\label{cortamodulo}
Let $V$ be an $\ff\ls{X}$--vector subspace of $\ff^n\ls{X}$. Then $V \cap \ff^n[X]$ is an $\ff[X]$--direct summand of $\ff^n[X]$. Moreover, $V$ is a convolutional code if and only if $V = (V \cap \ff^n[X])^*$.
\end{lemma}

Given a subset \(B \subseteq \F^n[X]\), we denote by \(\overline{B}\) the smallest \(\F[X]\)--direct summand of \(\F^n[X]\) that contains \(B\).

\begin{lemma}\emph{(\cite[Lemma 2]{Gomez/Sanchez:2024})}\label{extmodulo}
If $B \subseteq \ff^n[X]$, then $\overline{B} = B^* \cap \ff^n[X]$. Moreover,  the rank of the \(\ff[X]\)--module $\ff[X] B$ is equal to the dimension over $\ff\ls{X}$ of $B^*$. 
\end{lemma}

The following extends \cite[Lemma 4]{Gomez/Sanchez:2024} to the current setting. 

\begin{lemma}\label{estrellamodulo}
If \(B\) is any right \(A\)-submodule of \(\F^{n}\ls{X}\), then \(B^{*}\) is a right \(A\)-submodule of \(\F^{n}\ls{X}\).
\end{lemma}

\begin{proof} The elements of \(B^{*}\) are of the form \(b_1f_1+\cdots+b_tf_t\) where \(f_k\in \F\ls{X}\) and \(b_k\in B\) for \(1\leq k \leq t\). So, it suffices if we prove that \(bfa\in B^{*}\) for every \(a\in A\), \(b\in B\) and 
\(f=\sum_{i\geq i_0}f_iX^{i}\in \F\ls{X}\) with \(i_0\in \mathbb{Z}\). We consider two cases for \(s\): the sum of all monomials with non negative degrees and the sum of all monomials with negative degrees. First, let \(f=\sum_{i\geq 0}f_iX^{i}\in \F\s{X}.\)

Then, since \(\delta\) is nilpotent and by \eqref{lr infty},

\[bfa=\left(b\sum_{ i\geq 0}f_iX^{i}\right)a=\left(\sum_{i\geq 0}b\left(\sum_{j=i}^{(i+1)m-1}f_jN^{j}_{i}(a)\right)\right)X^{i}.\]
 
 Since \(A\) is finite dimensional over \(\F\) we have our \(\F\)-basis \(\{a_1,...,a_r\}\). Then \(\sum_{j=i}^{(i+1)m-1}N^{j}_{i}(a)f_j=\sum_{k=1}^{r}f_{ik}a_k\) for some \(f_{ik}\in \F\) for each \(i\geq 0\). Therefore
 
\[
bfa=\sum_{i\geq 0}b\left(\sum_{k=1}^{r}f_{ik}a_k\right)X^{i}=\sum_{k=1}^{r}ba_k\sum_{i\geq 0}f_{ik}X^{i}\in B^{*},
\]
since \(B\) is a \(A\)-submodule.

Now, consider \(f=\left( \sum_{i\geq i_0}^{-1}f_iX^{i}\right)\) for some  \(i_0 \leq -1\)  By \((\ref{x -m})\) and since \(B\) is a right \(A\)-submodule, it is straightforward that  \(bf_{i}X^{i}a\in B^{*}\) for each  \(i_0 \leq i\leq -1\). Hence \(bfa\in B^{*},\) as we required.
\end{proof}

\begin{theorem}\label{correspondencia}
The map \(\mathcal{D} \mapsto \mathcal{D} \cap \F^n[X]\) is a bijection between the set of all \((\sigma,\delta,A)\)--cyclic convolutional codes of rate \(k/n\) an the set of all right \(A[X;\sigma,\delta]\)--submodules of \(\F^n[X]\) that are \(\F[X]\)--direct summands of rank \(k\) of \(\F^n[X]\). 
\end{theorem}
\begin{proof}
A \((\sigma,\delta,A)\)--cyclic convolutional code \(\mathcal{D}\) is a right \(A\ls{X;\sigma,\delta}\)--submodule of \(\F^n\ls{X}\). Since \(A[X;\sigma,\delta]\) is a subring of \(A\ls{X;\sigma,\delta}\), one easily gets that \(\mathcal{D} \cap \F^[X]\) is an \(A[X;\sigma,\delta]\)--submodule of \(\F^n[X]\). By Lemma \ref{cortamodulo}, \(\mathcal{D} \cap \F^n[X]\) is an \(\F[X]\)--direct summand of \(\F^n[X]\). Now, \(\mathcal{D}\) has an \(\F\ls{X}\)--basis \(B\) whose elements belong to \(\F^n[X]\). Observe that, according to Lemma \ref{extmodulo}, \[\mathcal{D} \cap \F^n[X] = \overline{B} = \overline{B\F[X]},\] 
and the dimension \(k\) of the \(\F\ls{X}\)--vector space \(B^* = \mathcal{D}\) equals to the rank of the \(\F[X]\)--module  \(B\F[X]\). Finally, this rank is, precisely, the rank of \(\overline{B\F[X]}\). 

By virtue of Lemma \ref{cortamodulo} and Lemma \ref{extmodulo},  the reciprocal of the map described above,  sends a right \(A[X;\sigma,\delta]\)--submodule \(\mathcal{C}\) of \(\F^n[X]\) that is an \(\F[X]\)--direct summand onto \(\mathcal{C}^*\). Note that \(\mathcal{C}^*\) is a right \(A\)--submodule of \(\F^n\ls{X}\) by Lemma \ref{estrellamodulo}.
\end{proof}

Recall from \cite{Lopez/Szabo:2013} that a right ideal code is defined as a right ideal of \(A[X;\sigma,\delta]\) that it is an \(\F[X]\)--direct summand. Consider the regular right structure on \(A\) and any \(\F\)--vector space isomorphism \(A \cong \F^n\), with \(n\) the dimension of \(A\) over \(\F\). Then \(\F^n\) becomes a right \(A\)--module isomorphic to \(A\).

\begin{corollary}\label{idealcodes}
The map \(\mathcal{D} \mapsto \mathcal{D} \cap A[X;\sigma,\delta]\) is a bijection between the set of all \((\sigma,\delta,A)\)--cyclic convolutional codes in \(A\ls{X;\sigma,\delta}\) and the set of all right ideal codes. 
\end{corollary}
\begin{proof}
 In view of Lemma \ref{DSmodulo}, we may apply Theorem \ref{correspondencia} to get the bijection described in the statement of the Corollary. 
\end{proof}

Let us describe some concrete examples where the former cyclic convolutional codes can be built.

\begin{example} \label{examplenil} Let \(\F_4\) be the field with eight elements described as \(\F_4=\F_2(a)\), where \(a^2+a+1=0\). Let \(\tau\) be the Frobenius automorphism on \(\F_{8}\), that is, \(\tau(c)=c^2\) for every \(c\in \F_4\).   Set \(A=\mathcal{M}_2(\F_4)\), the ring of \(2\times 2\) matrices over $\F_4$, and consider the $\F_2$-automorphism $\sigma: A \to A$ defined as the component-by-component extension of $\tau$ to $A$. That is, $\sigma$ is given by
\begin{equation}\label{automorphism}
\sigma\left ( \begin{matrix} x_0 & x_1 \\ x_2 & x_3 \end{matrix} \right ) = \left ( \begin{matrix} \tau(x_0) & \tau(x_1) \\ \tau(x_2) & \tau(x_3) \end{matrix} \right ) = \left ( \begin{matrix} x_0^2 & x_1^2 \\ x_2^2 & x_3^2 \end{matrix} \right ) \text{ for every } \left ( \begin{matrix} x_0 & x_1 \\ x_2 & x_3 \end{matrix} \right )\in A. 
\end{equation}
We can also set the inner $\sigma$-derivation $\delta:A\rightarrow A$ given by $\delta(X)=MX-\sigma(X)M$ for $X\in A$, where 
\[
M=\begin{pmatrix} 
0 & 1 \\
0 & 0
\end{pmatrix}.
\]

 For any $x_0,x_1,x_2,x_3 \in \F_4$,
 \[
\begin{split}
\delta\begin{pmatrix} 
x_0 & x_1 \\
x_2 & x_3 
\end{pmatrix} &= \begin{pmatrix} 
0 & 1 \\
0 & 0
\end{pmatrix} \begin{pmatrix} 
x_0 & x_1 \\
x_2 & x_3 
\end{pmatrix} + \begin{pmatrix} 
x_0^2 & x_1^2 \\
x_2^2 & x_3^2 
\end{pmatrix}\begin{pmatrix} 
0 & 1 \\
0 & 0
\end{pmatrix}\\ 
&= \begin{pmatrix} 
x_2& x_3 \\
0 & 0 
\end{pmatrix} + \begin{pmatrix} 
0 &  x_0^2 \\
0 & x_2^2 
\end{pmatrix}\\
&= \begin{pmatrix} 
x_2& x_0^2 +x_3  \\
0 & x_2^2 
\end{pmatrix}.
\end{split}
\]

\[
\begin{split}
\delta^2\begin{pmatrix} 
x_0 & x_1 \\
x_2 & x_3 
\end{pmatrix} &= \begin{pmatrix} 
0 & 1 \\
0 & 0
\end{pmatrix} \begin{pmatrix} 
x_2& x_0^2 +x_3  \\
0 & x_2^2 
\end{pmatrix} + \begin{pmatrix} 
x_2^2& x_0 +x_3^2  \\
0 & x_2 
\end{pmatrix}
\begin{pmatrix} 
0 & 1 \\
0 & 0
\end{pmatrix}\\ 
&= \begin{pmatrix} 
0& x_2^2  \\
0 & 0 
\end{pmatrix} + \begin{pmatrix} 
0 &  x_2^2\\
0 & 0
\end{pmatrix}\\
&= \begin{pmatrix} 
0& 2x_2^2  \\
0 & 0 
\end{pmatrix}= \begin{pmatrix} 
0& 0 \\
0 & 0
\end{pmatrix}.
\end{split}
\]

Then \(\delta\) is nilpotent. Notice that \(\sigma(M)=M\), so \(\delta \sigma = \sigma\delta\) and thus, \(\delta'\) is nilpotent. Hence \(A\ls{X;\sigma,\delta}\) exists.
\end{example}

\begin{example}{\cite[Example 5.1]{Lopez/Szabo:2013}} \label{example:lopez/szabo}
Let \(\F_4=\F_2(a)\) and \(A = \F G\) the group algebra of a cyclic group \(G\) of order \(5\). Define \(\sigma : G \to G\) via \(\sigma(a)=a^2\), which is a group automorphism on \(G\). By extending \(\sigma\) linearly over \(\F\), \(\sigma\) gives  an \(\F\)--algebra automorphism on \(A\). Now, let \(\delta\) be the inner \(\sigma\)-derivation induced by \(1\), i.e.,  \(\delta(a)=a+\sigma(a)\). Let us check that \(\delta\) is nilpotent. Indeed, for any \(a \in G\),
\[
\begin{array}{rcl}
\delta(a)&=& a+a^2 \\
\delta^2(a)&=& a+a^2+a^2+a^2 =  a+a^4 \\
\delta^3(a)&=& a+a^4 + a^2+a^8=   a + a^4 + a^2 + a^3\\
\delta^4(a)&=& a + a^2  + a^3  + a^4 +  a^2  + a^4  + a^6   + a^8  = 0.\\
\end{array}
\]
Since \(\sigma(1)= 1\), one gets that \(\sigma \delta = \delta \sigma\) and, thus, \(\delta'\) is also nilpotent. Therefore,  \(A\ls{X; \sigma, \delta}\) exists, and the codes described in  \cite[Example 5.1]{Lopez/Szabo:2013}  are \((\sigma,\delta,A)\)--CCCs, by Theorem \ref{correspondencia}.
\end{example}

The following example shows that, even when working with the same ring, the same automorphism \(\sigma\), and an inner \(\sigma\)-derivation, it does not follow that \(\delta\) and \(\delta'\) are nilpotent.

\begin{example} Consider the assumptions of Example \ref{examplenil}. Now  take \[
M=\begin{pmatrix} 
0 & 0 \\
0 & a
\end{pmatrix},
\]
then

\[
\begin{split}
\delta\begin{pmatrix} 
0 & 0 \\
0 & a 
\end{pmatrix} &= \begin{pmatrix} 
0 & 0 \\
0 & a
\end{pmatrix} \begin{pmatrix} 
0 & 0\\
0 & a 
\end{pmatrix} + \begin{pmatrix} 
0 & 0 \\
0 & a^2
\end{pmatrix}\begin{pmatrix} 
0 & 0 \\
0 & a
\end{pmatrix}\\ 
&= \begin{pmatrix} 
0& 0 \\
0 & a^2
\end{pmatrix} + \begin{pmatrix} 
0 &  0 \\
0 & 1
\end{pmatrix}\\
&= \begin{pmatrix} 
0& 0  \\
0 & a
\end{pmatrix}.
\end{split}
\]
and, since \(\sigma^{-1}=\sigma\), we have that

\[
\begin{split}
\delta'\begin{pmatrix} 
0 & 0 \\
0 & a
\end{pmatrix} &=  \begin{pmatrix} 
0& 0  \\
0 & a 
\end{pmatrix}.
\end{split}
\]

Then neither \(\delta\) nor \(\delta'\) is nilpotent. 

\end{example}

Our final example shows that the nilpotence of \(\delta\) needs not imply that of \(\delta'\). 

\begin{example} Let \(A=\F[Y,Z]/I\) where  \(\F\) is a field and \(I\) is the ideal of the commutative polynomial ring \(\F[Y,Z]\) generated by \(\{Y^2, Z^2, YZ\}\). Set \(y = Y + I, z = Z+I\) and \(\sigma\) the \(\F\)--algebra automorphism of \(A\) defined by \(\sigma(1)=1, \sigma(y)=y+z\) and \(\sigma(z)=z\).
 Let us define the \(\F\)--linear map \(\delta\) as \(\delta(1)=0, \delta(y)=0\) and \(\delta(z)=y\).  Let us verify  that \(\delta\) is a \(\sigma\)-derivation:
\[
\begin{array}{l}
\sigma(y)\delta(y)+ \delta(y)y = (y+z)0+0y=0=\delta(y^2); \\
\sigma(z)\delta(z)+ \delta(z)z = zy+yz=0=\delta(z^2); \\
\sigma(y)\delta(z)+ \delta(y)z = (y+z)y +0y=0=\delta(yz); \\
\sigma(z)\delta(y)+ \delta(z)y = y0 +yy=0=\delta(zy), \\
\end{array}
\] 
as we required. Clearly,  \(\delta^2 =0\).  

Now, 
\[\delta'(y)= -\delta(\sigma^{-1}(y))=-\delta(y-z)= y.\]
Hence \(\delta'\) is not nilpotent. 
\end{example}

\section{Comments from the polynomial-based perspective.}\label{cabras}
We have addressed a specific conceptual question: given a cyclic structure on a convolutional code with polynomial coefficients, does it extend to (or arise from) a cyclic structure on the corresponding Laurent series model? 
We formulate these structures in terms of a left skew derivation \((\sigma,\delta)\) on a finite-dimensional algebra \(A\) over a field \(\F\).  While these approaches are strictly equivalent when \(\sigma\) is an automorphism and \(\delta = 0\) (see \cite{Gomez/Sanchez:2024}), Theorem \ref{correspondencia} shows that the answer remains positive when $\delta \neq 0$, provided that a suitable ring structure can be defined on the vector space of Laurent series extending the skew polynomial ring \(A[X;\sigma,\delta]\). The existence of such a ring structure, denoted by \(A\ls{X;\sigma,\delta}\), is discussed in detail in Section \ref{skewlaurentseries} and is of independent interest from a purely algebraic perspective.

Convolutional codes are widely used and remain the subject of extensive ongoing research (\cite{Alfarano/alt:2023,Climent/alt:2025,Gassner/alt:2024,Lieb/alt:2025,Pan/alt:2025}). These works typically adopt the polynomial-based approach. Notably, some of these codes are used to design crypto-systems, as seen in \cite{Gassner/alt:2024, Lieb/alt:2025}. Specifically,  \cite{Gassner/alt:2024} employs quasi-cyclic convolutional codes that, when non-catastrophic, fall within of cyclic convolutional codes defined in Section \ref{biyeccion}. In light of the above, Definition \ref{CCC} should be complemented with the following one.

\begin{definition}
Let \((\sigma,\delta)\) be a skew left derivation of an \(\F\)--algebra \(A\). Given a right \(A\)--module structure on \(\F^n\), endow \(\F^n[X]\) with a structure of right \(A[X;\sigma,\delta]\)--module according to the rule \eqref{AModuloDerecha}. A  \emph{polynomial \((\sigma,\delta,A)\)--cyclic convolutional code} is a right \(A[X;\sigma,\delta]\)--submodule of \(\F^n[X]\). 
\end{definition}

Convolutional quasi-cyclic codes from \cite{Gassner/alt:2024} are \((\mathrm{id},0,\F[t]/(t^r-1))\)--cyclic codes, where \(id\) denotes the identity map. By further specializing parameters to  \((id, 0, \F)\), we recover the standard notion of a convolutional code found in recent literature. However, it is often assumed that the code is an \(\F[X]\)--direct summand of \(\F^n[X]\). 

In this context, Theorem \ref{correspondencia} establishes that 
\((\sigma,\delta,A)\)-cyclic convolutional codes are in one-to-one correspondence to polynomial 
\((\sigma,\delta,A)\)-cyclic convolutional codes that are 
\(\mathbb{F}[X]\)-direct summands, provided that \(\sigma\) is an automorphism 
of a finite-dimensional \(\mathbb{F}\)-algebra \(A\) and both \(\delta\) and 
\(\delta' = -\delta \sigma^{-1}\) are nilpotent. This correspondence enables us to demonstrate how the general cyclic structure is reflected in the minimal encoder, in the spirit of \cite{Roos:1979}. We formulate this idea more precisely below.

Consider a convolutional code \(\mathcal{D} \subseteq \mathbb{F}^n[X]\). Following the framework established by \cite{Roos:1979}, given an encoder \(G\) defined by an \(\mathbb{F}[X]\)--basis \(\{g_1, \dots, g_k\} \subseteq \mathbb{F}^n[X]\) of \(\mathcal{D}\), we define its \emph{complexity} as
\[
    c(G) = \sum_{i = 1}^k \deg(g_i),
\]
where \(\deg(g_i)\) denotes the polynomial degree of \(g_i\). We call an encoder \emph{minimal} if it minimizes this complexity. Set \(\mathcal{C} = \mathcal{D} \cap \F^n[X]\). For any non-negative integer \(m\), define
\[
\mathcal{C}_m = \{ f \in \mathcal{C} : \deg(f) \leq m \}. 
\]
These \(\mathbb{F}\)-vector subspaces \(\mathcal{C}_m\) form a filtration of 
\(\mathcal{C}\), with \(\mathcal{C} = \bigcup_{m \geq 0} \mathcal{C}_m\). 
Consequently, for sufficiently large \(m\), we have 
\(k = \operatorname{rank}(\mathcal{C}) = \operatorname{rank}(\mathcal{C}_m \mathbb{F}[X])\), where \(\mathcal{C}_m \mathbb{F}[X]\) denotes the \(\mathbb{F}[X]\)-submodule generated by the subspace \(\mathcal{C}_m\). Following \cite{Roos:1979}, the \emph{span} of $\mathcal{D}$  is defined as
 \(\mu(\mathcal{D}) = \min \{ m : \mathcal{C}_m^* = \mathcal{D} \}.\) Alternatively, by \cite[Proposition 3]{Gomez/Sanchez:2024}, 
 \[ \mu(\mathcal{D}) = \min \{ m : \operatorname{rank}(\mathcal{C}_m \mathbb{F}[X]) = k \}.
 \]  
  Assume that \(\mathcal{D}\) 
is \((\sigma,\delta,A)\)-cyclic; then, \(\mathcal{C} = \mathcal{D} \cap \mathbb{F}^n[X] \) inherits the structure of a 
right \(A[X;\sigma,\delta]\)-submodule of \(\mathbb{F}^n[X]\). 
Proceeding as in \cite{Gomez/Sanchez:2024}, let \(\mathcal{C}_{-1} = \{0\}\). 
Then, for any \(m \in \mathbb{N}\), we have the inclusion
\[
    \mathcal{C}_{m-1} + \mathcal{C}_{m-1}X \subseteq \mathcal{C}_m.
\]
Observe that, as a consequence of \eqref{ga}, each of these \(\mathbb{F}\)-vector subspaces is a right 
\(A\)-submodule of \(\mathcal{C}\). If \(A\) is a semisimple algebra, 
there exists an \(A\)-submodule \(\mathcal{B}_m \subseteq \mathcal{C}_m\) such that
\begin{equation}\label{mdesc}
    (\mathcal{C}_{m-1} + \mathcal{C}_{m-1}X) \oplus \mathcal{B}_m = \mathcal{C}_m.
\end{equation}
 
 The following theorem extends \cite[Theorem 3]{Gomez/Sanchez:2024} to the case of 
skew derivations \((\sigma,\delta)\) with \(\delta\) not necessarily \(0\), assuming \(A\) is 
semisimple. The proof proceeds analogously to that of \cite[Theorem 3]{Gomez/Sanchez:2024} 
and is therefore omitted. A key modification consists of replacing 
\cite[Lemma 4]{Gomez/Sanchez:2024} with Lemma \ref{estrellamodulo}.

\begin{theorem}\label{directsumR}
    Assume that \(A\) is a semisimple \(\mathbb{F}\)-algebra, \(\sigma\) is an 
    automorphism of \(A\), and both \(\delta\) and \(\delta'\) are nilpotent. 
    Let \(\Gamma_m\) be an \(\mathbb{F}\)-basis of \(\mathcal{B}_m\) for every 
    \(m \geq 0\). Then, the union \(\Gamma_0 \cup \cdots \cup \Gamma_{\mu(\mathcal{D})}\) 
    forms a minimal encoder of \(\mathcal{D}\) and is an \(\mathbb{F}[X]\)-basis 
    of \(\mathcal{C} = \mathcal{D} \cap \mathbb{F}^n[X]\). Moreover,
    \[
        \mathcal{C} = \Gamma_0 \mathbb{F}[X] \oplus \cdots \oplus \Gamma_{\mu(\mathcal{D})}\mathbb{F}[X]
    \]
    is a direct sum decomposition of \(\mathcal{C}\) into \(A[X;\sigma,\delta]\)-submodules.
\end{theorem}

The following consequence of Theorem \ref{directsumR} is a general version of \cite[Theorem 4]{Roos:1979}. 

\begin{corollary}\label{alaRoos}
    The \((\sigma,\delta,A)\)-cyclic convolutional code \(\mathcal{D}\) admits a decomposition
    \[
        \mathcal{D} = \Gamma_0^* \oplus \cdots \oplus \Gamma_{\mu(\mathcal{D})}^*
    \]
    into a direct sum of \((\sigma, \delta, A)\)-cyclic convolutional subcodes.
\end{corollary}

The following example identifies what is arguably the simplest class of cyclic convolutional codes not covered by \cite{Gomez/Sanchez:2024} to which the theory developed in this paper applies.

\begin{example}\label{nilderivation}
Let \(A = M_r(\F)\) be the simple finite-dimensional \(\F\)-algebra of all square matrices of order \(r\) with coefficients in \(\F\). For any nilpotent matrix \(\alpha \in A\), the inner derivation \(\delta_\alpha : A \to A\), defined by \(\delta_\alpha(a) = \alpha a - a \alpha\) for all \(a \in A\), is nilpotent. With \(\sigma = id\), we may define the skew polynomial ring \(A[X;id,\delta_\alpha]\), denoted by \(A[X;\delta_\alpha]\) for simplicity. Furthermore, since  \(\delta_\alpha' = - \delta_\alpha\) is also nilpotent, the ring \(A\ls{X;\delta_\alpha}\) exists. Thus, \((id,\delta_\alpha, M_r(\F))\)-cyclic convolutional codes have a mathematically sound definition. Theorems \ref{correspondencia} and \ref{directsumR} and Corollary \ref{alaRoos} are applicable to these cyclic convolutional codes.
\end{example}

Theorem \ref{correspondencia} paves the way for proving that, under suitable 
assumptions, every \((\sigma,\delta,A)\)-cyclic convolutional code is generated 
by an idempotent skew polynomial. We focus on the case where \(\mathbb{F}^n\) is isomorphic to 
\(A\) as a right \(A\)-module. In this setting, \((\sigma,\delta,A)\)-cyclic 
convolutional codes correspond to the right ideals of \(A\ls{X;\sigma,\delta}\) 
that admit a generating set within \(A[X;\sigma,\delta]\). Corollary 
\ref{idealcodes} establishes that each of these codes \(\mathcal{D}\) is 
uniquely determined by the right ideal code
\(\mathcal{C} = \mathcal{D} \cap A[X;\sigma,\delta]\). When \(A\) is a separable algebra---which is always the case for a semisimple 
algebra over a finite field---the existence of an idempotent generator for 
\(\mathcal{C}\) is addressed in \cite{Gomez/alt:2017}. 

More precisely, let \(A \otimes A\) denote the tensor product of two copies of \(A\) over 
\(\mathbb{F}\), and define the maps \(\sigma^\otimes = \sigma \otimes \sigma\) 
and \(\delta^\otimes = \delta \otimes \mathrm{id} + \sigma \otimes \delta\). If there exists a 
separability idempotent \(e \in A \otimes A\) such that \(\sigma^\otimes (e) = e\) 
and \(\delta^\otimes (e)= 0\), then, by \cite[Theorems 6 and 8]{Gomez/alt:2017}, 
\(\mathcal{C}\) is generated by an idempotent \(\epsilon \in A[X;\sigma,\delta]\). 
Since \(\mathcal{D} = \mathcal{C}^*\) by Lemma \ref{cortamodulo}, it follows 
that \(\mathcal{D} = \epsilon A\ls{X;\sigma,\delta}\). Moreover, such an 
idempotent \(\epsilon\) can be computed using the method described in 
\cite[Section V]{Gomez/alt:2017}. Furthermore, a suitable separability element 
\(e \in A \otimes A\) (if it exists) can be found via the algorithm provided 
in \cite{Gomez/alt:2017b}. Finally, information regarding the free distance 
of \(\mathcal{C}\)---and consequently of \(\mathcal{D}\)---may be deduced from 
the generating idempotent \(\epsilon\) (see \cite{Gomez/alt:2021}).


\begin{thebibliography}{}
\bibitem{Alfarano/alt:2023}
G. N. Alfarano, D. Napp, A. Neri, and V. Requena. \emph{Weighted Reed–Solomon convolutional codes}, Linear Multilinear Algebra, \textbf{72(5)} (2023) 84--874.
\bibitem{Bergen/alt:2011}
J. Bergen and P. Grzeszczuk. \emph{Skew power series rings of derivation type.} J.
Algebra Appl. \textbf{10(6)} (2011), 1383--1399.
\bibitem{Climent/alt:2025} 
J. J. Climent,  D. Napp, and V. Requena, \emph{An algorithm to compute a minimal input-state-output representation of a convolutional code}, Linear Algebra Appl.,  \textbf{721} (2025), 715-735.
\bibitem{Cohn:1995} P. M. Cohn. \emph{Skew Fields: Theory of General Division Rings}, Encyclopedia of Mathematics and its Applications 57, Cambridge University Press, Cambridge, 1995.
\bibitem{Gassner/alt:2024}
N. Gassner, A. Mazumder, J. Rosenthal, and A. Sutton. \emph{An Approach to Constructing Convolutional Codes With Moderate Density and Quasi-Cyclic Structure}. IFAC-PapersOnLine, \textbf{58} (2024), 304--309.
\bibitem{Gluesing/Schmale:2004}
 H. Gluesing-Luerssen and W. Schmale. \emph{On cyclic convolutional codes.}
Acta Appl. Math., \textbf{82} (2004), 183--237.
\bibitem{Gomez/alt:2017}  J. G\'omez-Torrecillas, F.J. Lobillo and G. Navarro. \emph{Ideal codes over separable ring extensions}, IEEE Trans. Inform. Theory \textbf{63} (2017), 2796--2813. 
\bibitem{Gomez/alt:2017b} J. G\'omez-Torrecillas, F.J. Lobillo and G. Navarro. \emph{Computing separability elements for thesentence-ambient algebra of split ideal codes}, J. Symbolic Comput. \textbf{83} (2017), 211--227.
\bibitem{Gomez/alt:2021}  J. G\'omez-Torrecillas, F.J. Lobillo and G. Navarro. \emph{Cyclic distances of idempotent convolutional codes}, J. Symbolic Comput. \textbf{102} (2021), 37--62.
\bibitem{Gomez/Sanchez:2024} J. G\'omez-Torrecillas and J.P. S\'anchez-Hern\'andez. \emph{Minimal encoders of convolutional codes with general cyclic structures}, J. Algebra Appl. \textbf{23(10)} (2024) 2450164 (13 pages). 
\bibitem{Greenfeld/alt:2019}
B. Greenfeld, A. Smoktunowicz, and M. Ziembowskil. \emph{Five solved problems on radicals
of ore extensions.} Publ. Mat. \textbf{63} (2019), 423--444.
\bibitem{Lieb/alt:2025}
J. Lieb, R. Pinto and C. Vela,  
 \emph{A new method for erasure decoding of convolutional codes}. Des. Codes Cryptogr., (2025), 1-22.
\bibitem{Lopez/Szabo:2013}
S. R. L\'opez-Permouth and S. Szabo. 
 \emph{Convolutional codes with additional algebraic structure.}
 J. Pure Appl. Algebra, \textbf{217} (2013), 958--972.
 \bibitem{Pan/alt:2025}
 X. Pan, H. Chen, C. Tang and X. Chen,  \emph{New Bounds for Generalized Column Distances and Construction of Convolutional Codes}, IEEE Trans. Inform. Theory,  \textbf{71(4)}, (2025), 2576--2590
\bibitem{Piret:1976}
P. Piret. \emph{Structure and constructions of cyclic convolutional codes}, IEEE Trans. Inform. Theory, \textbf{22} (1976), 147--155
\bibitem{Roos:1979}
C. Roos. \emph{On the Structure of Convolutional and Cyclic Convolutional Codes}, IEEE Trans. Inform. Theory, \textbf{25} (1979), 676--683. 
\bibitem{Schneider/Venjakob:2005}
P. Schneider and O. Venjakob, \emph{On the codimension of modules over skew power series rings with applications to Iwasawa algebras.}  
J. Pure Appl. Algebra \textbf{204} (2006), no. 2, 349--367.
\bibitem{Schneider/Venjakob:2010}
P. Schneider and O. Venjakob, \emph{Localizations and completions of skew power series rings.}
American J. Math. \textbf{132} (2010), 1--36.
\bibitem{Stenstrom:1975}
B. Stenstr{\"o}m. Rings of Quotients, Springer, 1975.
\end{thebibliography}
\end{document}